\documentclass[12pt]{amsart}

\usepackage{amssymb}
\usepackage{bm}
\usepackage{graphicx}
\usepackage{psfrag}
\usepackage[dvipsnames]{xcolor}
\usepackage{enumerate}
\usepackage{multirow}
\usepackage{url}
\usepackage{hyperref}

\usepackage{subcaption}
\usepackage{comment}

\usepackage{algpseudocode}

\usepackage{mathtools}

\usepackage{xy}
\input xy
\xyoption{all}

\usepackage{tikz}

\newcommand{\excise}[1]{}

\newtheorem{thm}{Theorem}[section]
\newtheorem{lemma}[thm]{Lemma}

\newtheorem{cor}[thm]{Corollary}
\newtheorem{prop}[thm]{Proposition}

\theoremstyle{definition}

\newtheorem{example}[thm]{Example}
\newtheorem{remark}[thm]{Remark}
\newtheorem{defn}[thm]{Definition}

\numberwithin{equation}{section}

\newcommand\RR{\mathbb{R}}

\newcommand\ZZ{\mathbb{Z}}

\def\wt#1{{\widetilde {#1}}}
\newcommand\set[1]{\left\{#1\right\}}

\DeclareMathOperator\lcm{lcm} 

\newcommand{\CEND}{C_\textnormal{END}}

\newcommand\nd{\mathrm{nd}}
\newcommand\gnd{\mathrm{gnd}}

\begin{document}

\mbox{}
\title{Eventually nondecreasing quasi-polynomials}

\author[B.~Braun]{Benjamin Braun}
\address{Mathematics Department\\University of Kentucky\\Lexington, KY 40506}
\email{benjamin.braun@uky.edu}

\author[C.~O'Neill]{Christopher O'Neill}
\address{Mathematics Department\\San Diego State University\\San Diego, CA 92182}
\email{cdoneill@sdsu.edu}

\author[A.~Park]{Antwon Park}
\address{Mathematics Department\\University of Kentucky\\Lexington, KY 40506}
\email{antwon.park@uky.edu}

\date{21 July 2026}

\thanks{BB and AP were partially supported by NSF award DMS-2450299.}

\begin{abstract}
Quasi-polynomials are ubiquitous in combinatorics and algebra, as they arise in a variety of enumeration problems.
Because quasi-polynomials consist of constituent polynomials, their behavior is more subtle than for a single polynomial.
In particular, unlike for a polynomial, it is possible for a quasi-polynomial defined on the positive integers to have infinitely many points at which it is decreasing.
In this work, we characterize quasi-polynomials of degree $d$ and period dividing $p$ that are eventually nondecreasing, i.e., that have only finitely many values at which they decrease.
We then give a detailed analysis of the space of eventually nondecreasing quasi-polynomials with fixed degree $d$ and period dividing a fixed $p$ such that the $h$-vector of the quasi-polynomial is nonnegative.
Using this analysis, we determine the rate of growth of the number of such quasi-polynomials as a function of the sum of the $h$-vector entries for the $0$-th constituent polynomial.
\end{abstract}

\maketitle


\section{Introduction}
\label{sec:intro}

Quasi-polynomials play a key role in combinatorics and algebra, as they frequently arise as the solution to enumeration problems related to positive integer sequences.
For example, quasi-polynomials are the integer-point enumerators for dilates of rational polytopes~\cite{ccd}.
They also arise as Hilbert series of graded algebras that are not standard graded~\cite{brunsherzog,numerical}.
Quasi-polynomials also play a major role in the study of Presburger arithmetic~\cite{presburgerarith,woodsunreasonable}.

Because quasi-polynomials consist of a finite set of constituent polynomials, their behavior is more subtle than for a single polynomial.
In particular, unlike for a polynomial, it is possible for a quasi-polynomial defined on the positive integers to have infinitely many points at which it is decreasing.	
The motivation for this work is to classify and study those quasi-polynomials that, like a polynomial with a positive leading coefficient, are eventually nondecreasing.

For a quasi-polynomial $f(t)$ defined on $\ZZ_{\geq 0}$, we use the standard encoding of $f$ in a rational generating function of the form
\[
\sum_{n \ge 0} f(n) z^n
:= \frac{h_0 + h_1z + h_2z^2 + \cdots +h_{p(d+1)}z^{p(d+1)}}{(1-z^p)^{d+1}}\, .
\]
Throughout combinatorics and algebra, especially for enumeration problems involving graded Cohen-Macaulay algebras, the vector $(h_0,\ldots,h_{p(d+1)})$ is a major focus of study.
Thus, we frame our investigations as the study of $h$-vectors for quasi-polynomials, rather than as the study of the coefficients of the individual constituent polynomials.
With this point of view, it becomes natural to study quasi-polynomials for which the $h$-vector satisfies special properties, such as nonnegativity.

Our main contributions in this work are as follows.
\begin{itemize}
\item We introduce in Section~\ref{sec:notation} our primary objects of study, the $N$-th nondecreasing cone $C_N$ and the eventually nondecreasing cone $\CEND$, each of which is defined for quasipolynomials of a fixed degree $d$ with period dividing a fixed $p$.
\item We characterize in Section~\ref{sec:endcone} the set of $h$-vectors that belong to $\CEND$.
We also prove that the topological closure of $\CEND$, denoted $\overline{\CEND}$, is a polyhedral cone for every fixed $d$ and $p$.
\item In Section \ref{sec:generalpolyhedrality}, we prove that $C_N$ is polyhedral if and only if $d = 1$ or $d = 2$.
\item Motivated by properties of quasi-polynomials arising as Ehrhart series and Hilbert series of Cohen-Macaulay algebras, in Section \ref{sec:nondecreasingend} we consider the cone of nonnegative vectors in $\overline{\CEND}$.
We classify the ray generators of the nonnegative part of $\overline{\CEND}$, and in Section~\ref{sec:sufficiency} we establish sufficient conditions to detect when one of those vertices is in $\CEND$.
\item In Section~\ref{sec:growth}, we study the enumeration of quasipolynomials with nonnegative integer $h$-vectors by an invariant that we call the volume (motivated by Ehrhart theory).
We prove that for fixed $d$ and $p$, as a function of volume the number of eventually nondecreasing quasipolynomials is itself a quasipolynomial.
\end{itemize}

Some of the proofs of our main results are long and technical, and for those proofs we include them in Section~\ref{sec:proofs}.

\section{Notation and setup}
\label{sec:notation}

Throughout this article, fix $d, p \in \ZZ_{\ge 0}$ with $p \ge 1$.  
For $h = (h_0, h_1, h_2, \ldots)\in \RR^{p(d+1)}$, let 
\[
h_{r,j} := h_{jp + r}
\qquad \text{for} \qquad
0 \le r < p
\qquad \text{and} \qquad
0 \le j \le d \, ,
\]
and let $h_{r,*} := (h_{r,0}, h_{r,1}, \cdots) \in \RR^{d+1}$ for each $r$.
Define $h(z) := h_0 + h_1z + h_2z^2 + \cdots$, and define $L_h:\ZZ_{\ge 0} \to \RR$ via
\[
\sum_{n \ge 0} L_h(n) z^n
:= \frac{h(z)}{(1-z^p)^{d+1}}\, .
\]
By expanding $1/(1-z^p)^{d+1}$ as a power series, it is straightforward to verify that for each $0\leq r<p$, we have 
\begin{equation}
L_h(tp + r) = \sum_{j = 0}^d h_{r,j} \binom{t+d-j}{d}= \sum_{j = 0}^d  \frac{h_{r,j}}{d!}\prod_{k=0}^{d-1}(t+d-j-k)\, . \label{eq:qpdef} 
\end{equation}
Using the right-hand side of~\eqref{eq:qpdef}, we see that $L_h(tp+r)$ is a polynomial in $t$ of degree at most $d$.
Thus, for each choice of $h\in \RR^{p(d+1)}$,  the function $L_h$ is a quasi-polynomial of degree at most $d$ and period dividing $p$.

Note that when $t$ and $r$ are both fixed, and when the values $h_{r,j}$ are treated as variables, then $L_h(tp+r)$ is a linear expression in $h_{r,*}$.
Thus, the condition 
\[
L_h(n+1) \ge L_h(n) \, ,
\]
which captures that $L_h$ is nondecreasing at $n$, yields a linear inequality on the entries of the vector $h$.
As such, in order for the function $L_h$ to be nondecreasing at all but finitely many input values, $h$ must satisfy the above inequality for all but finitely many~$n$ (note that this is an infinite collection of linear inequalities on the entries of $h$).  

\begin{example}
\label{ex:hlinear}
For $d=3$, $p=4$, and $r=1$, we have
\begin{align*}
L_h(21) & = L_h(5p+r)
\\
&=h_{1,0}\binom{5+3-0}{3}+h_{1,1}\binom{5+3-1}{3}+h_{1,2}\binom{5+3-2}{3}+h_{1,3}\binom{5+3-3}{3} \\
&= 56h_{1,0}+35h_{1,1}+20h_{1,2}+10h_{1,3}\, . 
\end{align*}
Thus, $L_h(21)\geq L_h(20)$ yields the inequality
\[
56h_{1,0}+35h_{1,1}+20h_{1,2}+10h_{1,3}\geq 56h_{0,0}+35h_{0,1}+20h_{0,2}+10h_{0,3}\, ,
\]
while  $L_h(24)\geq L_h(23)$ yields the inequality
\[
84h_{0,0}+56h_{0,1}+35h_{0,2}+20h_{0,3}\geq 56h_{3,0}+35h_{3,1}+20h_{3,2}+10h_{3,3}\, .
\]
\end{example}

This motivates the following definition.  
Recall that a \emph{(convex) cone} is a subset $C \subseteq \RR^d$ that is closed under non-negative linear combinations.

\begin{defn}\label{d:nondecreasingcone}
For each $N \in \ZZ_{\ge 0}$, the \emph{$N$-th nondecreasing cone} is given by
\[
C_N := \{h \in \RR^{p(d+1)} : L_h(n+1) \ge L_h(n) \text{ for all } n \ge N\} \, .
\]
Note that $C_0 \subseteq C_1 \subseteq C_2 \subseteq \cdots$.
We define the \emph{always nondecreasing cone} to be $C_\textnormal{AND} := C_0$ and the \emph{eventually nondecreasing cone} to be $C_\textnormal{END} = \bigcup_{N \ge 0} C_N$.
We say a quasi-polynomial $L_h$, or equivalently the vector $h$, is \emph{always nondecreasing} if $h\in C_0$.
We say $L_h$ and $h$ are \emph{eventually nondecreasing} if $h \in C_\textnormal{END}$.
\end{defn}

The following lemma verifies the sets $C_N$ and $C_\textnormal{END}$ are actually cones.

\begin{lemma}
The set $C_N$ is a cone for each $N\geq 0$, as is $C_\textnormal{END}$.  
\end{lemma}

\begin{proof}
For each $s \ge 0$, as stated previously one may view
$L_h(n+1) \ge L_h(n)$
as a linear inequality on the entries of $h$.  
As such, each $C_N$ equals the solution set of a (countably infinite) system of linear inequalities, and thus is a cone.  Additionally, since $C_0 \subseteq C_1 \subseteq C_2 \subseteq \cdots$, their union $C_\textnormal{END}$ is also a cone.  
\end{proof}

\section{Characterizing eventually nondecreasing quasi-polynomials}%
\label{sec:endcone}

Our goal in this section is to characterize the vectors $h \in C_\textnormal{END}$.
The key observation is the following:\ since $L_h(n)$ is a quasi-polynomial in $n$ of degree at most $d$ and period dividing $p$, the same can be said for $L_h(n+1) - L_h(n)$, and the former is eventually nondecreasing if and only if the latter is eventually nonnegative.  Since a quasi-polynomial is eventually nonnegative if and only if each nonzero constituent has positive leading coefficient, we may determine whether $L_h(n)$ is eventually non-decreasing by computing the leading coefficients of the constituents of $L_h(n+1) - L_h(n)$.  Each coefficient of a given constituent of $L_h(n+1) - L_h(n)$ is a linear function of the entries of $h$, but since the leading coefficient is also dependent on the degree of its constituent, the resulting characterization of $C_\textnormal{END}$ involves a sequence of checks.  Note also that the coefficients of $L_h(n)$, when expressed in terms of the entries of $h$, depend on Stirling numbers, which notoriously do not possess a closed form; thankfully, after some modest algebra, this complexity can be avoided when checking if $h$ is eventually nondecreasing.

\begin{thm}\label{t:endh}
We have $h \in C_\textnormal{END}$ if and only if the following hold:
\begin{enumerate}[1.]
\item 
for each $r \in [0, p-2]$, either for each $\ell \in [0, d]$ the quantity
$$
A^{(\ell)}_r(h) := \sum_{j=0}^d (-j)^\ell h_{r+1,j} - \sum_{j=0}^d (-j)^\ell h_{r,j}
$$
vanishes, or the smallest $\ell$ for which $A^{(\ell)}_r(h) \ne 0$ has $A^{(\ell)}_r(h) > 0$; and
\item 
either for each  $\ell \in [0,d]$ the quantity
$$
B^{(\ell)}(h) := \sum_{j=0}^d (-j)^\ell h_{0,j}
- \sum_{j=0}^d (-j-1)^\ell h_{p-1,j}
$$
vanishes, or the smallest $\ell$ for which $B^{(\ell)}(h) \ne 0$ has $B^{(\ell)}(h) > 0$.  
\end{enumerate}
\end{thm}

\begin{proof}
Let $e_{d,k}(x_1,\ldots,x_d)$ denote the degree~$k$ elementary symmetric polynomial on $d$~variables and let $E_{d,k} := e_{d,k}(1,2,\ldots,d)$. 
We first observe that the following equality holds for any $0\leq j\leq d$, as it will be used throughout the proof:
\begin{align*}
e_{d,k}(1-j,\ldots,d-j) 
&= \sum_{\substack{A\subseteq[d]\\ |A|=k}} \, \prod_{\alpha\in A}(-j+\alpha)
= \sum_{\substack{A\subseteq[d] \\ |A|=k}} \, \sum_{\ell=0}^k(-j)^{\ell}\sum_{\substack{B\subseteq A \\ |B|=k-\ell}} \, \prod_{b\in B}b
\\
&= \sum_{\ell=0}^k (-j)^{\ell} \!\! \sum_{\substack{B \subseteq [d] \\ |B|=k-\ell}} \, \sum_{\substack{B\subseteq A \\ |A|=k}} \, \prod_{b\in B}b
= \sum_{\ell=0}^k (-j)^{\ell} \binom{d-k+\ell}{\ell} \sum_{\substack{B\subseteq [d] \\ |B|=k-\ell}} \, \prod_{b\in B}b
\\
&= \sum_{\ell=0}^k(-j)^{\ell}\binom{d-k+\ell}{\ell}E_{d,k-\ell}.
\end{align*}
Suppose $r \in [0, p-1]$.  
Observe that
\begin{align*}
L_h(tp + r)
&= \sum_{j = 0}^d h_{r,j} \binom{t+d-j}{d}
= \frac{1}{d!} \sum_{j = 0}^d h_{r,j} \sum_{k = 0}^d e_{d,k}(1-j, \ldots, d-j) t^{d-k}
\\
&= \frac{1}{d!} \sum_{j = 0}^d h_{r,j} \sum_{k = 0}^d \bigg( \sum_{\ell = 0}^k \binom{d-k+\ell}{\ell} E_{d,k-\ell} \ (-j)^\ell \bigg) t^{d-k}
\\
&= \frac{1}{d!} \sum_{k = 0}^d \sum_{\ell = 0}^k \binom{d-k+\ell}{\ell} E_{d,k-\ell} \bigg(\sum_{j = 0}^d h_{r,j} (-j)^\ell \bigg) t^{d-k}.
\end{align*}
We first consider the case where $r \ne p-1$.
Suppose that $A^{(\ell)}_r(h)=0$ for all $\ell<m$, or equivalently that
$$
\sum_{j=0}^d (-j)^\ell h_{r+1,j}
= \sum_{j=0}^d (-j)^\ell h_{r,j}
\qquad \text{for all} \qquad
\ell < m.
$$
Therefore,
\begin{align*} 
L_h(tp + r + 1) - L_h(tp + r)
&= \frac{1}{d!} \sum_{k = 0}^d \sum_{\ell = 0}^k \binom{d-k+\ell}{\ell} E_{d,k-\ell} \bigg(\sum_{j = 0}^d (h_{r+1,j} - h_{r,j}) (-j)^\ell \bigg) t^{d-k}
\end{align*}
has degree at most $d-m$ since the innermost sum vanishes for $\ell < m$.
Further, the coefficient of $t^{d-m}$ equals
\begin{equation}\label{eq:A_leading_coeff}
	\frac{1}{d!} \binom{d}{m} \sum_{j = 0}^d (h_{r+1,j} - h_{r,j}) (-j)^m
	= \frac{1}{d!} \binom{d}{m} A^{(m)}_r(h).
\end{equation}
Therefore, the $r$-th constituent of $L_h(n+1) - L_h(n)$ is eventually nonnegative if and only if $A^{(m)}_r(h)>0$.

We next consider the case $r=p-1$, which will be proved similarly.
Suppose  $B^{(\ell)}(h)=0$ for all $\ell<m$, or equivalently, that
$$
\sum_{j=0}^d (-j)^\ell h_{0,j}
= \sum_{j=0}^d (-j-1)^\ell h_{p-1,j}
\qquad \text{for all} \qquad
\ell < m.
$$
Setting $E_{d,k}' = e_{d,k}(2, \cdots, d+1)$ and $e_{d,k} = 0$ unless $0 \le k \le d$, observe that
\begin{align*}
L_h(tp + p - 1)
&= \frac{1}{d!} \sum_{j = 0}^d h_{p-1,j} \sum_{k = 0}^d e_{d,k}(1-j, \cdots, d-j) t^{d-k}
\\
&= \frac{1}{d!} \sum_{k = 0}^d \sum_{\ell = 0}^k \binom{d-k+\ell}{\ell} E_{d,k-\ell}' \sum_{j = 0}^d h_{p-1,j} (-j-1)^\ell t^{d-k}
\\
&= \frac{1}{d!} \sum_{k = 0}^d \sum_{\ell = 0}^k \sum_{n = 0}^{k-\ell} \binom{d-k+\ell+n}{n} \binom{d-k+\ell}{\ell} E_{d,k-\ell-n} \sum_{j = 0}^d h_{p-1,j} (-j-1)^\ell t^{d-k}.
\end{align*}
Additionally,
\begin{align*}
L_h(tp + p)
&= L_h((t+1)p)
\\
&= \frac{1}{d!} \sum_{s = 0}^d \sum_{\ell = 0}^s \binom{d-s+\ell}{\ell} E_{d,s-\ell} \sum_{j = 0}^d h_{0,j} (-j)^\ell (t+1)^{d-s}
\\
&= \frac{1}{d!} \sum_{s = 0}^d \sum_{\ell = 0}^s \binom{d-s+\ell}{\ell} E_{d,s-\ell} \sum_{j = 0}^d h_{0,j} (-j)^\ell \sum_{n = 0}^{d-s} \binom{d-s}{n} t^{d-s-n}
\\
&= \frac{1}{d!} \sum_{k = 0}^d \sum_{n = 0}^k \sum_{\ell = 0}^{k-n} \binom{d-k+n+\ell}{\ell} \binom{d-k+n}{n} E_{d,k-\ell-n} \sum_{j = 0}^d h_{0,j} (-j)^\ell t^{d-k}
\end{align*}
where the substitution $k = s + n$ is used in the last step.  In the final expressions for $L_h(tp + p)$ and $L_h(tp + p - 1)$, all 4 sums may be extended to an upper limit of $d$ without introducing any nonzero summands, so the coefficient of $t^{d-k}$ in the expression for 
\begin{align*}
L_h(tp + p) - L_h(tp + p - 1)
\end{align*}
is then given by
\begin{align*}
\frac{1}{d!} \sum_{k = 0}^d \sum_{n = 0}^d \sum_{\ell = 0}^d \binom{d-k+n+\ell}{n} \binom{d-k+\ell}{\ell} E_{d,k-\ell-n} \bigg(\sum_{j = 0}^d h_{0,j}(-j)^\ell - h_{p-1,j}(-j-1)^\ell \bigg).
\end{align*}
As before, for $k < m$ the coefficient of $t^{d-k}$ is zero, and the coefficient of $t^{d-m}$ has identical sign to $B^{(m)}(h)$. 
Therefore, the remaining constituent of $L_h(n+1) - L_h(n)$ is eventually nonnegative if and only if $B^{(m)}(h) > 0$.
\end{proof}

\begin{cor}\label{cor:disjointend}
The cone $\CEND$ is a disjoint union of the relative interiors of finitely many polyhedral cones.  
\end{cor}

\begin{proof}
Fix nonnegative integers $a_0, a_1, \dots, a_{p-2}, b$ and define the set $C_{a_0,a_1,\dots,a_2, b}$ to be
\[
\{h \in \RR^{p(d+1)} : A_r^{(i)} = 0 \text{ for all } i \leq a_r, A_r^{(a_r+1)} > 0, B^{(i)} = 0 \text{ for all } i \leq b, B^{(b+1)} > 0\} \, .
\]
Observe that $C_{a_0,a_1,\dots,a_2, b}$ is the intersection of finitely many hyperplanes and an open half-space, hence is a relatively open polyhedral cone.
Further, by Theorem~\ref{t:endh}, we have $C_{a_0,a_1,\dots,a_{p-2},b} \subset \CEND$.

We claim that $\CEND$ is the disjoint union of the sets $C_{a_0,a_1,\dots,a_{p-2},b}$.
Since there are finitely many possible equations of the form $A_r^{(i)} = 0$ or $B^{(i)}=0$, we have only finitely many possible values for $(a_0,\dots,a_{p-2},b)$.
Thus, there are finitely many sets in the union.
We know that the union of the $C_{a_0,a_1,\dots,a_{p-2},b}$ is contained in $\CEND$.
Further, consider $h \in C_{a_0,a_1,\dots,a_{p-2},b}$ for a fixed set of subscript values. 
Then the zero condition on the $A$ and $B$ functions combined with $A_{r}^{(a_r + 1)} > 0$ for all $r$ and $B^{(b+1)} > 0$ yield that $h$ cannot be an element of any other set of the form $C_{a_0,a_1,\dots,a_{p-2},b}$, hence the union is disjoint.

Finally, for $h\in \CEND$, the indices of the $A$ and $B$ equations for which the inequalities of Theorem~\ref{t:endh} are strict yield that $h$ is contained in some $C_{a_0,a_1,\dots,a_{p-2},b}$.
Thus, the disjoint union is actually equal to $\CEND$, and our proof is complete.
\end{proof}

Observe that for any $h\in C_{\textnormal{END}}$, the non-negativity of each $A_r^{(0)}(h)$ and $B^{(0)}(h)$ in Theorem~\ref{t:endh} yields the inequalities
$$
\sum_{j=0}^d h_{0,j}
\le \sum_{j=0}^d h_{1,j}
\le \cdots
\le \sum_{j=0}^d h_{p-1,j}
\le \sum_{j=0}^d h_{0,j},
$$
which together imply the above quantities must coincide for any $h\in C_\textnormal{END}$.  
This implies that the minimum $\ell$ in the statement of Theorem~\ref{t:endh} satisfies $\ell \ge 1$ when it exists.  
Hence, for the $\ell = 1$ case of Theorem~\ref{t:endh}, we must have $A_r^{(1)}(h) \ge 0$ for each $r$ and $B^{(1)}(h) \ge 0$.
These two observations motivate the following.

\begin{defn}\label{def:volumefirstlevel}
We define the \emph{volume equalities} in $\RR^{p(d+1)}$ by 
\[
\sum_{j=0}^d h_{r+1,j}
= \sum_{j=0}^d h_{r,j}
\qquad \text{for each} \qquad
r \in [0, p-2] \, .
\]
Given a vector $h\in \RR^{p(d+1)}$ that satisfies the volume equalities, we define the \emph{volume} of $h$ to be $\sum_{j=0}^{d}h_{0,j}$.
We define the \emph{first-level} inequalities in $\RR^{p(d+1)}$ by
\[
\sum_{j=0}^d jh_{r+1,j}
\le \sum_{j=0}^d jh_{r,j}
\quad \text{for each} \quad
r \in [0, p-2]
\]
and
\[
\sum_{j=0}^d jh_{0,j}
\le \sum_{j=0}^d (j+1)h_{p-1,j}.
\]
\end{defn}

\begin{remark}\label{rem:expectation}
Consider the case where for each $0\leq r\leq p-1$, we have $\sum_{j=0}^dh_{r,j}=1$ and $h_{r,j}\geq 0$ for all $j$.
In this case, the $h$-vector consists of $p$ finite probability distributions on $\{0,1,\ldots,d\}$.
The volume equalities imply that, in the case of non-negative $h$-coefficients, if one mod class is a probability distribution then they all are.
The first-level inequalities in this case assert that the expected values of the distributions weakly decrease, and the inequality involving $p-1$ and $0$ imply that all of the expected values are contained in a line segment of length $1$.
This observation will play a key role in Section~\ref{sec:nondecreasingend}.

More generally, observe that the inequalities $A_r^{(\ell)}(h)$ are comparisons of $\ell$-th moments of the distributions in different mod classes.
Theorem~\ref{t:endh} states that in order for $h$ to be in the end cone, we must have that for each $r$, the first time that a pair of higher moments are not equal, the values of the two moments are ordered in a particular way.
The order required for each pair of higher moments depends on the parity of $\ell$. 
The inequality determined by $B^{(\ell)}$ when $\ell\geq 2$ is more complicated, due to the $(-j-1)^{\ell}$ coefficient, which results in the involvement of the $i$-th moments of $h_{p-1,*}$ for $0\leq i\leq \ell$.
\end{remark}

\begin{cor}\label{c:endclosure}
The topological closure $\overline{C_\textnormal{END}}$ is a $(pd + 1)$-dimensional polyhedral cone in $\RR^{p(d+1)}$ whose $H$-description is comprised of the volume equalities and first-level inequalities.  
\end{cor}

\begin{proof}
Let $K_{d,p} \subseteq \RR^{p(d+1)}$ denote the cone defined by the volume equalities and first-level inequalities, and let $K_{d,p}^\circ$ denote the relative interior of this cone.
Our proof strategy is to show that 
\[
K_{d,p}^\circ \subseteq C_\textnormal{END} \subseteq K_{d,p}
\]
which establishes our claim that $K_{d,p} = \overline{C_\textnormal{END}}$.

First, observe that all of these spaces satisfy the volume equalities.
Second, observe that $K_{d,p}^\circ$ is obtained by making all first-level inequalities strict, and thus any $h\in K_{d,p}^\circ$ satisfies the positivity requirement for $\ell=1$ in Theorem~\ref{t:endh}, implying $h\in C_\textnormal{END}$.
Also, by Theorem~\ref{t:endh}, any $h \in C_\textnormal{END}$ satisfies each of the first-level inequalities either strictly or with equality, and hence $h\in K_{d,p}$.
This completes our proof.
\end{proof}

We close this section with some geometric properties of $C_\textnormal{END}$ and $\overline{C_\textnormal{END}}$.

\begin{prop}\label{prop:linealitydimensions}
The lineality space of $C_\textnormal{END}$ is $1$-dimensional and the quotient of $\overline{C_\textnormal{END}}$ by its lineality space is a $p$-dimensional simplicial cone.
\end{prop}

\begin{proof}
Any $h$ in the lineality space of $C_\textnormal{END}$ must have $-h \in C_\textnormal{END}$ as well.
Thus, there exists some $N$ such that for all $n\geq N$, we have both
\[
L_h(n+1)\geq L_n(n)
\]
and
\[
L_h(n+1)=-L_{-h}(n+1)\leq -L_{-h}(n) = L_{h}(n) \, .
\]
This means $L_h(n)$ must be eventually constant, which implies it is constant since each constituent of $L_n(n)$ is a polynomial.  
As such, the lineality space of $C_\textnormal{END}$ consists of only constant functions, and is thus $1$-dimensional. 

We next consider $\overline{C_\textnormal{END}}$.  
By Corollary~\ref{c:endclosure}, any $h$ in the lineality space of $\overline{C_\textnormal{END}}$ must satisfy every first-level inequality with equality, in addition to satisfying every volume equation; this yields $2p - 1$ linear equations defining the lineality space.  
Adding the equation $A_0^{(1)}(h) = 0$ to the equation $B^{(1)}(h) = 0$ results in an upper-triangular system, so the lineality space of $\overline{C_\textnormal{END}}$ has dimension 
$$p(d+1) - (2p - 1) = p(d-1) + 1.$$
Lastly, since $\overline{C_\textnormal{END}}$ is contained in the subspace cut out by the volume equalities, its affine span has dimension $p(d+1) - (p-1) = dp + 1$.
Thus, the pointed image of $\overline{C_\textnormal{END}}$ in the quotient by its lineality space is a $p$-dimensional cone defined by $p$ linear inequalities, and hence is simplicial.
\end{proof}

\section{The polyhedrality of the cone $C_N$}
\label{sec:generalpolyhedrality}

In this section, our goal is to prove Theorem~\ref{thm:cnpolyhedral}, which states that the cone $C_N$ is polyhedral precisely when either $d=1$ or $d=2$.
As we will be considering the inequalities on $h$-vectors defined by $L_h(n+1)\geq L_h(n)$, throughout this section we will be utilizing the difference polynomial defined for $r \in [0,p-1]$ by 
\[
f_r(t) := L_h(tp+r+1) - L_h(tp+r).
\]
Observe that~\eqref{eq:qpdef} implies that $L_h(tp + r)$ is a degree $d$ polynomial in $t$ with degree $d$ coefficient $\tfrac{1}{d!} (h_{r,0} + h_{r,1} + \dots + h_{r,d})$. 
The proofs of Theorem~\ref{t:endh} and Corollary~\ref{c:endclosure} guarantee that for $h\in \CEND$, these degree $d$ coefficients will be equal for all $r$ (this is equivalent to $A^{(0)}_r(h) = 0$).
Thus, to determine if $f_r(t)$ is eventually nonnegative, we need to investigate its leading coefficient, which will be of the form $A^{(\ell)}_r(h)$ for some $\ell \neq 0$. 

We begin by proving polyhedrality when $d=1$ and $d=2$.
	 
\begin{prop}\label{prop:d1polyhedral}
	If $d=1$, then $C_0 = C_N = \overline{\CEND}$ for all $N \in \ZZ_{\geq 0}$.
\end{prop}
\begin{proof}
	Since $C_0 \subseteq C_N \subseteq \overline{\CEND}$, our proof will be complete once we show $\overline{\CEND} \subseteq C_0$.
	Let $h \in \overline{\CEND}$. 
	By~\eqref{eq:A_leading_coeff}, since $A_r^{(0)}(h) = B^{(0)}(h) = 0$ for each  for all $r \in [0,p-2]$, each $f_r(t)$ is constant.  Moreover, this constant is non-negative since each $A_r^{(1)}(h) \geq 0$, and $B^{(1)}(h) \geq 0$.  
	 Therefore, since each $f_r(t)$ is nonnegative, we have $h \in C_0$.
\end{proof}

\begin{prop}\label{prop:d2polyhedral}
	Let $d = 2$ and $N \in \ZZ_{\geq 0}$. Then $h \in C_N$ if and only if $h \in \overline{\CEND}$ and 
	\[L_h(N+i) \leq L_h(N+i+1) \quad \text{for} \quad i \in [0,p-1]\]
	In particular, $C_N$ is polyhedral.  
\end{prop}
\begin{proof}
	Observe that $C_N \subseteq \overline{\CEND}$ and that the claimed inequalities are among the defining inequalities for $C_N$.
	Hence, the forward directly is immediate.
	
	Assume $h \in \overline{\CEND}$ and $L_h(N+i) \leq L_h(N+i+1)$ for $i \in [0,p-1]$.  Fix $r \in [0,p-1]$.  Since $h \in \overline{\CEND}$, by~\eqref{eq:A_leading_coeff} we can write $f_r(t) = at + b$ with $a \ge 0$.  Supposing $tp + r \ge N$, and fixing $k \in \ZZ_{\ge 0}$ and $i \in [0,p-1]$ so that $tp + r - (N + i) = kp$, we have
	\[
		f_r(t) = at + b = ka + f_r(t-k) = ka + L_h(N+i+1) - L_h(N+i) \ge 0,
	\]
	as desired.
\end{proof}

For $d \geq 3$, it is possible to construct an $h$-vector such that $L_h(n)$ experiences a single decrease at an arbitrarily specified input value.  

\begin{example}\label{ex:general_single_decrease}
	Let $d = 3$ and $p = 4$ and suppose we want a single decrease $L_h(18) > L_h(19)$. Considering the difference polynomials $f_r(t)$, this can be done by setting $f_2(t) = (t - 3)(t - 5) = t^2 - 8t + 15$ and the remaining $f_r(t) = 0$. Thus, we have the following set of equations.
	\begin{align*}
		f_0(t) &= 0,  &	f_2(t) &= t^2 - 8t + 15, \\
		f_1(t) &= 0,  &	f_3(t) &= 0.
	\end{align*}
	Note that $f_r(t) < 0$ if and only if $r = 2$ and $t = 4$, which corresponds with $L_h(19) = L_h(4t+r+1) < L_h(4t+r) = L_h(18)$. It then suffices to find a solution to the following system
	\begin{align*}
		\sum_{j=0}^d (h_{1,j} - h_{0,j}) \binom{t+d-j}{d} &= 0 \\
		\sum_{j=0}^d (h_{2,j} - h_{1,j}) \binom{t+d-j}{d} &= 0 \\
		\sum_{j=0}^d (h_{3,j} - h_{2,j}) \binom{t+d-j}{d} &= t^2 - 8t + 15 \\
		h_{0,0} \binom{t+d+1}{d} - h_{3,d}\binom{t}{d} + \sum_{j=0}^{d-1} \left(h_{0,j+1} - h_{3,j}\right) \binom{t+d-j}{d} &= 0
	\end{align*}
	This system is underdetermined with a single degree of freedom. Let us set $h_{0,0} = 0$. Then expanding the left hand side of each equation, we get the following 16 equations after comparing coefficients:
	\begin{align*}
		h_{1,0}  &= 0 & h_{2,0} - h_{1,0} &= 0 & h_{3,0} - h_{2,0} &= 15 & h_{0,1} - h_{3,0} &= 0\\
		h_{1,1} - h_{0,1} &= 0 & h_{2,1} - h_{1,1} &= 0 & h_{3,1} - h_{2,1} &= -52 & h_{2,2} - h_{3,1} &= 0\\
		h_{1,2} - h_{0,2} &= 0 & h_{2,2} - h_{1,2} &= 0 & h_{3,2} - h_{2,2} &= 61 & h_{2,3} - h_{3,2} &= 0\\
		h_{1,3} - h_{0,3} &= 0 & h_{2,3} - h_{1,3} &= 0 & h_{3,3} - h_{2,3} &= -24 & - h_{3,3} &= 0
	\end{align*}
	This system has a unique solution, namely
	\[h = (0,0,0,15,15,15,15,-37,-37,-37,-37,24,24,24,24,0).\]
	And so here is an example such that $L_h(n+1) \geq L_h(n)$ for all $n$ except $n = 18$.
\end{example}

It is possible for us to generalize this construction. In particular, suppose $g_r(t)$ are eventually nondecreasing degree $d$ polynomials for each $r \in [0, p-1]$ and let us consider the difference polynomials $f_r(t) := L_h(tp+r+1) - L_h(tp+r)$, given by
\begin{align}
	f_r(t) &= \sum_{j=0}^d (h_{r+1,j} - h_{r,j}) \binom{t+d-j}{d} \quad \text{for} \quad  r \in [0,p-2],\label{eq:f_r(t)} \\
	f_{p-1}(t) &=  \sum_{j = 0}^d \left((-1)^j \binom{d+1}{j+1} h_{0,0} + h_{0,j+1} - h_{p-1,j}\right)\binom{t+d-j}{d} \label{eq:f_{p-1}(t)}
\end{align}
where $h_{0,d+1} := 0$. 
From the definition of $f_{p-1}(t)$, there should be a binomial term of $\binom{t+d+1}{d}$; however, one can show, using binomial identities or a combinatorial inclusion-exclusion argument, that \eqref{eq:f_{p-1}(t)} is an equivalent expression by using the binomial identity
\begin{equation}\label{eq:t+d+1}
	\binom{t+d+1}{d} = \sum_{j=0}^d (-1)^j \binom{d+1}{j+1} \binom{t+d-j}{d} \, .
\end{equation}

We would like to find some $h \in \RR^{p(d+1)}$ such that $f_r(t) = g_r(t)$ for all $r$, and by construction of $g_r(t)$, $h$ would belong in $\CEND$. Since $\set{\binom{t+d-j}{d}}_{j = 0}^d$ is a basis of polynomials of degree at most $d$, for each $r \in [0,p-1]$, there exists unique $\alpha_{r,j}$ such that 
\begin{equation}\label{eq:binomial_change_of_basis}
	\sum_{j=0}^d \alpha_{r,j} \binom{t+d-j}{d} = g_r(t)\, .
\end{equation}
In other words, for $r \neq p-1$, let us set $\alpha_{r,j} = h_{r+1, j} - h_{r,j}$ from \eqref{eq:f_r(t)}, and in addition, let us set 
\begin{equation}\label{eq:a_{p-1,j}}
	\alpha_{p-1,j} = h_{0,j+1} - h_{p-1,j} + (-1)^j\binom{d+1}{j+1}h_{0,0}\, .
\end{equation}

\begin{lemma}\label{l:specified_difference}
	Suppose that, for each $r \in [0,p-1]$, we have $g_r(t)$ is a polynomial of degree at most $d$, and fix $\alpha_{r,j}$ for each $j \in [0,d]$ such that~\eqref{eq:binomial_change_of_basis} holds. There exists $h \in \RR^{p(d+1)}$ such that $f_r(t) = g_r(t)$ for all $r$ if and only if 
	\[
		\sum_{j=0}^{d} \sum_{r=0}^{p-1} \alpha_{r,j} = 0.
	\]
\end{lemma}

\begin{proof}
	Suppose that $h$ exists such that $f_r(t) = g_r(t)$. From~\eqref{eq:f_r(t)}, if $r \in [0,p-2]$, we have that $h_{r+1,j} - h_{r,j} = \alpha_{r,j}$ for $j \in [0,d]$.
	Additional, from \eqref{eq:a_{p-1,j}}, we have that $\alpha_{p-1,j} = h_{0,j+1} - h_{p-1,j} + (-1)^j\binom{d+1}{j+1}h_{0,0}$. 
	
	Adding each of the $\alpha_{r,j}$ yields
	\begin{align*}
		\sum_{j=0}^{d} \sum_{r=0}^{p-1} \alpha_{r,j} &= \sum_{j=0}^d \left(h_{0,j+1} - h_{p-1,j} + (-1)^j \binom{d+1}{j+1} h_{0,0} + \sum_{r=0}^{p-2} (h_{r+1,j} - h_{r,j})\right)\\
		&=\sum_{j=0}^d \left(h_{0,j+1} - h_{p-1,j} + (-1)^j \binom{d+1}{j+1} h_{0,0} + h_{p-1,j} - h_{0,j}\right)\\
		&= \sum_{j=0}^d \left(h_{0,j+1} + (-1)^j \binom{d+1}{j+1} h_{0,0} - h_{0,j}\right)\\
		&= -h_{0,0} + \sum_{j=0}^d  (-1)^j \binom{d+1}{j+1} h_{0,0}\\
		&= \sum_{j=0}^{d+1}  (-1)^{j-1} \binom{d+1}{j} h_{0,0} = 0\, .
	\end{align*}
	
	Now suppose that $\alpha_{r,j}$, as defined in \eqref{eq:binomial_change_of_basis}, are such that the sum of all $\alpha_{r,j}$ is 0. Then we would like to find some $h$ such that $f_r(t) = g_r(t)$ for all $r \in [0,p-1]$. Such an $h$ can be constructed by setting $h_{0,0} = 0$ and $h_{r,j}= \sum_{i=0}^{jp+r-1} \alpha_i$, where $\alpha_{r,j} = \alpha_{jp+r}$, for all other entries.

	Since $h_{0,0} = 0$, for all $r \in [0,p-1]$ and $j \in [0,d]$, $\alpha_{r,j}$ can be expressed as the difference $h_{r+1,j} - h_{r,j}$, where we consider $h_{p,j} = h_{0, j+1}$ and $h_{0,d+1} = 0$. If $r \neq p-1$ or $j \neq d$, then we have $h_{r+1,j} - h_{r,j} = \alpha_{r,j}$, and otherwise, by our assumption $\sum_{i=0}^{pd+p-1} \alpha_i = 0$, we can conclude $h_{p-1,d} = -\alpha_{p-1,d}$.
	Therefore, we have for all $r \in [0,p-1]$,
	\[f_r(t) = \sum_{j = 0}^d (h_{r+1,j} - h_{r,j}) \binom{t+d-j}{d} = \sum_{j = 0}^d \alpha_{r,j} \binom{t+d-j}{d} = g_r(t).\]
\end{proof}

\begin{lemma}\label{l:sum_alpha}
	Let $r \in [0,p-1]$ be fixed and let $g_r(t)$ be a polynomial of degree at most $d$. Fix $\alpha_{r,j}$ for each $j \in [0,d]$ such that \eqref{eq:binomial_change_of_basis} holds.
	If $a_d$ is the coefficient of $t^d$ in $g_r(t)$, then $\sum_{j=0}^d \alpha_{r,j} = d! a_d $.
\end{lemma}
\begin{proof}
	The coefficient of $t^d$ in $\binom{t+d-j}{d}$ is $\frac{1}{d!}$, thus comparing coefficients in \eqref{eq:binomial_change_of_basis}, we obtain
	\[a_d = \sum_{j = 0} \frac{\alpha_{r,j}}{d!},\]
	which gives our desired result.
\end{proof}

The previous two lemmas then give rise to the existence of some $h \in \overline{\CEND}$ for $d \geq 3$ such that $L_h(n) > L_h(n+1)$ for any desired $n$, as shown in the proof of the following theorem.

\begin{thm}\label{thm:cnpolyhedral}
 The cone $C_N$ is polyhedral if and only if $d = 1$ or $d=2$.
\end{thm}

\begin{proof}
We have already shown that $C_N$ is polyhedral for $d=1$ and $d=2$. We will next prove that for $d\geq 3$, $C_N$ is not polyhedral. 
We will show that each of the inequalities of the form $L_h(n+1) \geq L_h(n)$ for $n \geq N$ is necessary to describe $C_N$.
Specifically, for every choice of $k \in \ZZ_{\geq 0}$ and $m \in [0,p-1]$, we will produce an $h$ such that $L_h(n+1) \geq L_h(n)$ for all values of $n$ except for $n = kp + m$.

Suppose $f_m(t) = t^2 - 2kt +k^2 -1 = (t-k-1)(t-k+1)$ and $f_r(t) = 0$ for all $r \neq m$. We know there exists $\alpha_{r,j}$ such that
\[f_r(t) = \sum_{j=0}^d \alpha_{r,j} \binom{t+d-j}{d}.\]
Then by Lemma \ref{l:sum_alpha}, since $d \geq 3$ and each difference polynomial is degree at most 2, we have that the coefficients on the degree $d$ term is 0, hence $\alpha_{r,0} + \alpha_{r,1} + \dots + \alpha_{r,d} = 0$. Therefore, by Lemma \ref{l:specified_difference}, there exists $h \in \RR^{p(d+1)}$ such that 
\[L_h(tp+r+1) - L_h(tp+r) = \begin{cases}
	t^2 - 2kt + k^2 -1 & r = m\\
	0 & \text{otherwise}.
\end{cases}\]
Note that $L_h(kp+ m + 1) - L_h(kp+m) < 0$ if and only if $t = k$ and $r = m$. Therefore, there is no finite collection of halfspaces that will define $C_N$.
\end{proof}

\section{Nonnegative $h$-vectors in $\overline{\CEND}$}
\label{sec:nondecreasingend}

Many interesting examples of quasi-polynomials arise in commutative algebra and combinatorics, often in the context of $h$-vectors for Hilbert series of graded Cohen-Macaulay algebras equipped with non-standard $\ZZ$-gradings.
In such cases, all coefficients of the $h$-vector are nonnegative~\cite{brunsherzog}.
This motivates our study of the space of nonnegative $h$-vectors in $\CEND$ and $\overline{\CEND}$.
To do so, we will consider the rational polytope
\[
S_1 :=  \overline{\CEND} \cap \left\{h \in \RR_{\geq 0}^{p(d+1)} : \sum_{j=0}^d h_{0,j} = 1\right\} 
\]
that is obtained by restricting the cone $ \overline{\CEND}\cap \RR_{\geq 0}^{p(d+1)}$ to the hyperplane $\sum_{j=0}^d h_{0,j} = 1$, which by the volume equalities implies that $\sum_{j=0}^d h_{i,j} = 1$ for all $i$.
We will focus our attention on describing the rays of $\overline{\CEND}\cap \RR_{\geq 0}^{p(d+1)}$ and determining which of these rays are contained in $\CEND$. 
Note that the rays of $ \overline{\CEND}\cap \RR_{\geq 0}^{p(d+1)}$ are generated by the vertices of $S_1$.

Since $S_1$ is the set of all nonnegative $h \in \overline{\CEND}$ such that the coordinate sum of each mod class (i.e. the volume of each mod class) is 1, we are naturally led to consider the following definition.

\begin{defn}\label{d:CM}
	Given $h \in \overline{\CEND}$ and $r \in [0,p-1]$, let $CM_r(h)$ denote the quantity
	\[CM_r(h) := \frac{\sum_{j = 0}^d j  h_{r,j}}{V},\]
	where $V = \sum_{j = 0}^d h_{0,j}$.
	We call $CM_r(h)$ the \emph{center-of-mass} of the $r$-th mod class of $h$.
	We denote the vector $CM(h) := (CM_0(h), CM_1(h), \dots, CM_{p-1}(h))$ and we call $CM(h)$ the \emph{center-of-mass vector} of $h$.
	When it is clear, we will abbreviate $CM_r$ and $CM$.
\end{defn}

The intuition behind this naming convention is that when the volume is $1$, the quantity in the first-order inequalities is the center-of-mass of a discrete system with total mass 1 where masses $h_{r,j}$ are located at positive integer distances from the origin. 
Note that the total mass of the system being 1 is captured by restricting ourselves to consider points in $S_1$.
Thus, we can write the first-level inequalities (when restricted to $S_1$) as
\begin{equation}\label{eq:first_order_CM}
	-1 + CM_0 \leq CM_{p-1} \leq CM_{p-2} \leq \dots \leq CM_1 \leq CM_0.
\end{equation}
As a matter of fact, many of the arguments through this section can be generalized to work for a point $h \in \overline{\CEND}$, rather than restricting to $S_1$. However, since $\overline{\CEND}$ is a cone, we can scale any point $h \in \overline{\CEND}$ such that the scaled point lies in $S_1$, and in addition, the volume of every point in $S_1$ being 1 allows us to simplify  much of the algebra, primarily from dropping the denominator in any computation regarding $CM_r$.

\begin{remark}\label{rem:cmprobability}
Another way to interpret the value $CM_r$ is as an expected value, as described in Remark~\ref{rem:expectation}.
Specifically, we know that in $S_1$ the sum of the $h_{r,j}$ for a fixed mod class is equal to $1$, and the values are all non-negative. 
Thus, each mod class encodes a probability distribution on $\{0,1,\ldots,d\}$. 
The expression for $CM_r$ is precisely the expected value of this distribution, and the first-level inequalities are bounds on the expected values for the $p$ probability distributions contained in $h$.
Finite probability distributions arising in this form have been the recent subject of study in Ehrhart theory~\cite{ehrhartdistributions}.
\end{remark}

Our main result of this section is the following characterization of the elements of $S_1$ that are vertices.

\begin{thm}\label{t:S1_vertex}
	Let $h \in S_1$. Then $h$ is a vertex of $S_1$ if and only if the following are true:
	\begin{enumerate}
		\item[(1)] $\max(h) = 1$,
		\item[(2)] there are at most two nonzero entries of $h$ in each mod class, and
		\item[(3)] $CM_r \in \ZZ$ for all $r \in [0,p-1]$
	\end{enumerate}
\end{thm}

The structure of our proof of Theorem~\ref{t:S1_vertex} is as follows: in this subsection, we will prove two lemmas and show that the reverse implication holds.
The forward direction of the proof will be obtained through three arguments by contrapositive, one for each of the three conditions in the characterization.
Because the arguments by contrapositive are longer and more technical, we defer those to Section~\ref{sec:proofs}.

The following lemma establishes a key relationship between a point in $S_1$ and its center-of-mass vector.
We first observe that, much like how at least one mass in a balanced system must be on each side of the center-of-mass, if there are multiple nonzero entries of $h$ in the $r$-th mod class, there must be one nonzero entry on each side of $CM_r(h)$. 
Suppose $r \in [0,p-1]$ and $h \in \RR^{p(d+1)}_{\geq 0}$ and let us define $\alpha$ and $\omega$ to be the smallest and largest values, respectively, such that $h_{r,\alpha} \neq 0$ and $h_{r,\omega} \neq 0$. 
Then if there are multiple nonzero entries, then $\alpha \neq \omega$ and since the nonnegativity of $h_{r,j}$ implies
\[CM_r = \sum_{j = \alpha}^\omega j h_{r,j} > \sum_{j= \alpha}^\omega \alpha h_{r,j} = \alpha 
\quad \text{and} \quad
CM_r = \sum_{j = \alpha}^\omega j h_{r,j} < \sum_{j= \alpha}^\omega \omega h_{r,j} = \omega \, ,
\]
we have
\begin{equation}\label{eq:CM_balance}
	\alpha < CM_r < \omega.
\end{equation}

\begin{lemma}\label{l:determined_h}
If $v \in S_1$ has at most two nonzero entries in each mod class, then the values $CM_r(v)$ uniquely determine $v$.
\end{lemma}
\begin{proof}
	We will consider $CM_r$ for a fixed $r$; extend this construction for each mod class to fully determine $v$.
	If $v$ has only one nonzero entry in the $r$-th mod class, then the volume equations force that entry to have a value of $1$.
	Since $CM_r(v)$ is given, it must be the case that $CM_r(v) \in \ZZ$ and $h_{r, CM_r(v)} = 1$.
	
	Otherwise, $v$ has two nonzero entries at $v_{r,\alpha}$ and $v_{r, \omega}$. 
	The vector $v$ must satisfy the equations
	\[\sum_{j=0}^d v_{r,j} = 1 \quad \text{and} \quad \sum_{j=0}^d jv_{r,j} = CM_r(v)\]
	which reduce to a system of two equations in two unknowns, with unique solution
	\begin{equation}\label{eq:determined_h}
		v_{r,j} = \begin{cases}
			\dfrac{\omega - CM_r}{\omega - \alpha} & j = \alpha\\[1em]
			\dfrac{CM_r - \alpha}{\omega - \alpha} & j = \omega\\[1em]
			0 & \text{otherwise} 
		\end{cases}\, .
	\end{equation}
\end{proof}

We next provide a useful result regarding the general form $CM(h)$ will take if $h \in \overline{\CEND}$ and $CM(h) \in \ZZ^p$.

\begin{lemma}\label{l:integer_CM}
	Let $h \in S_1$ such that $CM(h) \in \ZZ^p$. Then there exists some integer $k \in \ZZ$ such that
	\[CM(h) = (k+1,k+1,\dots,k+1, k,k,\dots,k).\]
	We define the \emph{offset index}, denoted $s$, to be the unique index such that $CM_s = k+1$ and $CM_{s+1} = k$. If no such index exists, i.e. $CM$ is a constant vector, we say $s = p-1$.
\end{lemma}
\begin{proof}
	Since $h \in S_1$, $h$ satisfies the first-level inequalities, and so we have the following string of inequalities,
	\[CM_0 \geq CM_1 \geq CM_2 \geq \dots \geq CM_{p-1} \geq CM_0 - 1.\]
	Thus, we have that $CM_{r-1} \geq CM_{r} \geq CM_{r-1} - 1$ for all $r > 0$, and so $CM \in \ZZ^p$ then implies either $CM_{r-1} = CM_{r}$ or $CM_{r-1} = CM_r + 1$. Let $s$ be the largest index such that $CM_{s} = CM_{s+1} + 1$. If no such index exists, then $CM$ is a constant vector, and we set $s = p-1$. Otherwise, $s$ exists, and for all $r > s$, we have 
	\[CM_{s+1} \geq CM_{r} \geq CM_{s} - 1 = CM_{s+1},\]
	thus $CM_r = CM_{s+1}$ for all $r > s$. Suppose there exists $s' < s$ such that $CM_{s'} = CM_{s'+1} + 1$. Then from the first-level inequalities, we would have 
	\[CM_{s'} = CM_{s' + 1} +1 \geq CM_{s} + 1 = CM_{s+1} + 2 \geq CM_{p-1} + 2 \geq CM_{0}+ 1 \geq CM_{s'} + 1, \]
	which is a contradiction. Therefore, we have that $CM_r = CM_{s}$ for all $r \geq s$, and hence $CM(h)$ is of the desired form.
\end{proof}

\begin{proof}[Proof of Theorem~\ref{t:S1_vertex}]
	We first consider the reverse direction.
	Thus, let $v \in S_1$ and assume that conditions (1), (2), and (3) hold.
	Note that by conditions (1) and (2), there exists $m \in [0,p-1]$ and $k \in [0,d]$ such that $v_{m,k} = 1$ and $v_{m,j}=0$ for all $j\neq k$, and hence we have that $CM_m(v) = k$ and $v$ has at most $2p - 1$ nonzero entries. 
	Therefore, $v$ has at least $dp - p + 1$ entries with a value of 0. Since $CM_r(v)\in\ZZ$ for all $r$, the first-level inequalities yield that for all $r$, either $CM_r(v) = CM_0(v)$ or $CM_r(v) = CM_0(v) - 1$.
	Thus, all but one of the first-level inequalities will be satisfied with equality. 
	This implies that $v$ satisfies at least $p(d+1)$ facet-defining inequalities with equality, specifically, the $p$ volume equations of the form 
	\[
	\sum_{j=0}^d v_{r,j} = 1 \, ,
	\]
	the $dp - p + 1$ equations of the form $v_i = 0$, and $p-1$ equations either of the form $CM_r(v) = CM_{r+1}(v)$ or of the form $CM_{p-1}(v) = CM_{0}(v) - 1$.
	
	Now, suppose that a point $w\in S_1$ satisfies the same collection of facet-defining inequalities as $v$.
	Then $w_{m,j} =0$ for all $j\neq k$, which implies that $w_{m,k}=1$ and hence $CM_m(w) = k = CM_m(v)$.
	Since $v \in S_1 \subseteq \overline{\CEND}$, by Lemma \ref{l:integer_CM}, let $s$ be the offset index. Note that if $s \neq 0$, then $CM_{r-1}(v) = CM_{r}(v)$ for all positive $r \neq s$ and $CM_0(v) = CM_{p-1}(v) + 1$ and if $s = 0$, then $CM(v)$ is a constant vector.
	
	Thus, if $s \neq 0$, we have that $CM_{s-1}(v) \geq CM_s(v)$ is the only inequality among the first-level that is satisfied by $v$ with strict inequality, and if $s = 0$, then it is only $CM_{p-1}(v) \geq CM_{0}(v) - 1$ that is satisfied by $v$ with strict inequality. Therefore, since we know $CM_m(w) = CM_m(v)$ and $v$ and $w$ share the same offset index, we can use Lemma \ref{l:integer_CM} to show that $CM(w) = CM(v)$.
	Therefore, Lemma \ref{l:determined_h} implies that $w$ is unique, and hence $w=v$.
	Thus, $v$ is a vertex of $S_1$.
	
	The forward direction is proved by three contrapositive arguments, one for each condition in the theorem.
	Applying Lemma~\ref{l:max}, Proposition~\ref{p:2nonzero_decomp}, and Proposition~\ref{p:nonintegerCM} completes the proof of the forward direction.
\end{proof}

\begin{remark}
	Note that Theorem \ref{t:S1_vertex} and \eqref{eq:determined_h} provide a method to produce all the vertices of $S_1$. 
	Specifically, we know that some vertex $h$ of $S_1$ must have an entry equal to $1$ and we also know that we can partition the vertices by their $CM$ vector.
	Thus, we first fix some $CM$ vector of an appropriate form, given by Lemma \ref{l:integer_CM} and then, we considers all possible subcases arising from a choice of which mod classes will have two nonzero entries and where those two non-zero entries will be. Note that by condition (1) in Theorem \ref{t:S1_vertex}, we must select at least one mod class to have only one nonzero entry.
	Each of the resulting cases will result in a vertex, and every vertex arises via this process.
\end{remark}

\begin{example}\label{ex:vertexmethod}
	Let $d = 4$ and $p = 3$. 
	Consider the case $h_{1,2} = 1$, which forces $CM_1=2$.
	By the first-level inequalities, we have that 
	\[
	CM_0 \geq 2 \geq CM_2 \geq CM_0 -1 \, .
	\]
	Thus, $CM_0$ can be either $2$ or $3$.
	If $CM_0=2$, then $CM_2$ can be either $2$ or $1$.
	Alternatively, if $CM_0 = 3$, then $CM_2 = 2$; consider this subcase.
	For $r = 0$, if there are two nonzero entries, then the right nonzero entry must be at $h_{0,4}$ and either $h_{0,0}$ or $h_{0,1}$ can be the left nonzero entry. 
	Suppose that $h_{0,1}, h_{0,4}, h_{2,0}, h_{2,4}$ are the nonzero entries of $h$. 
	Then from (\ref{eq:determined_h}), we can compute that
	\[h = \begin{bmatrix}
		0 &  0  & \frac{2}{4}&
		\frac{1}{3} & 0 & 0 &
		0 & \underline{1} & \underline{0} & 
		\underline{0} & 0 & 0 &
		\frac{2}{3} & 0 & \frac{2}{4}
	\end{bmatrix}\]
	The entries at $h_{r, CM_r}$ are underlined, and we have deliberately expressed some of the coordinates as $\tfrac{2}{4}$ rather than the simplified $\tfrac{1}{2}$.
	Notice that each mod class is balanced about the underlined entry and the denominators are equal to the distance, i.e., $\omega - \alpha$, between the nonzero entries. 
	In general, vertices of $S_1$ take on this form where, in each mod class, if there are two nonzero entries, then they appear with denominators equal to the distance between them.
	To find all vertices in this case, one would complete the case-by-case analysis for all possible choices in this search process.
\end{example}

\section{When is a vertex of $S_1$ in $\CEND$?}\label{sec:sufficiency}

We can use Theorem \ref{t:S1_vertex} to classify the vertices of $S_1$.
However, for $d \geq 3$, by Theorem \ref{thm:cnpolyhedral}, we know that $\CEND \subsetneq \overline{\CEND}$. 
Thus, some of the vertices of $S_1$ may not belong to $\CEND$. 
In this section, we would like to identify simple properties of a vertex of $S_1$ that would determine whether or not it is in $\CEND$.
 
We know from the definition of $S_1$ that $A_r^{(0)} = 0$ and $B^{(0)} = 0$ for all points in $S_1$.
By the first-level inequalities, we have $A_r^{(1)} \geq 0 $ and $B^{(1)} \geq 0$ for all points in $S_1$. 
Our goal is to find various sufficient conditions on $h$ to establish that either $A_r^{(2)} < 0$ or $B^{(2)} < 0$ in the case of equality in one of the preceding equations.
As we also know that there are at most two nonzero entries in each mod class for any vertex of $S_1$, the following lemma will aide in computations.

\begin{lemma}\label{l:A2_computation}
	Let $r \in [0,p-1]$ be fixed. Suppose that $h$ has two nonzero entries in the $r$-th mod class, namely $h_{r,\alpha}$ and $h_{r,\omega}$. Then 
	\[\sum_{j=0}^d (-j)^2 h_{r,j} = (\alpha+\omega)CM_r - \alpha\omega.\]
\end{lemma}

\begin{proof}
	From Lemma \ref{l:determined_h}, we know the values of $h_{r,\alpha}$ and $h_{r,\omega}$, and we can show that
	\[ \sum_{j=0}^d (-j)^2 h_{r,j} = (-\alpha)^2 \frac{\omega - CM_r}{\omega - \alpha} + (-\omega)^2 \frac{CM_r - \alpha}{\omega-\alpha} = (\alpha+\omega)CM_r - \alpha\omega.\]
\end{proof}

Here we present some simple reject criteria of whether a vertex $h$ of $S_1$ belongs in $\CEND$, based on consecutive entries on $h$. A fact we will use throughout the proofs is that for a vertex $h \in S_1$, if $h_i = 1$ and $r \equiv i \pmod p$, then $i = p \cdot CM_{r} + r$.

\begin{prop}\label{p:END_01_check}
	Let $h$ be a vertex of $S_1$ and let $s$ be the offset index of $CM(h)$, as defined in Lemma \ref{l:integer_CM}. If there exists an index $i$ such that $i \not\equiv s \pmod p$, $h_{i} = 0$, and $h_{i+1} = 1$, then $h \notin \CEND$.
\end{prop}

\begin{proof}
	First, let $i$ such that $h_{i} = 0$ and $h_{i+1} = 1$ and let $r \in [0,p-1]$ be such that $i \equiv r \pmod p$. By our assumption, we have that $r \neq s$.
	We would like to show that $A_{r}^{(2)} < 0$ (or $B^{(2)} < 0$ if $r = p-1$). 
	For now, we will assume that $r \neq p-1$.
	Since $r \neq s$, we have that $CM_{r} = CM_{r+1}$ by the definition of $s$. Then by using Lemma \ref{l:A2_computation}, we have 
	\begin{align*}
		A_{r}^{(2)} &= \sum_{j=0}^d (-j)^2 h_{r+1,j} - \sum_{j=0}^d (-j)^2 h_{r,j}\\
		&=  CM_{r}^2 -  (\alpha + \omega)CM_{r} + \alpha\omega\\
		&= (CM_{r}-\alpha)(CM_{r}-\omega).
	\end{align*}
	By equation (\ref{eq:CM_balance}), we know $\alpha < CM_r < \omega$, and so $A_{r}^{(2)} < 0$.
	
	Now, let $r = p-1$. Since $r \neq s$, we have $CM_0 = CM_{p-1} + 1$. Thus, we can similarly show that 
	\begin{align*}
		B^{(2)} &= \sum_{j=0}^d (-j)^2 h_{0,j} - \sum_{j=0}^d (-j-1)^2 h_{p-1,j}\\
		&= (CM_{p-1}-1)^2 - (\alpha + \omega)CM_{p-1} + \alpha\omega\\
		&= (CM_{p-1}-\alpha)(CM_{p-1}-\omega) - 2CM_{p-1} + 1
	\end{align*}
	Since $0 \leq \alpha < CM_{p-1}$, by equation (\ref{eq:CM_balance}), we have that $CM_{p-1}$ is at least $1$ and so $-2CM_{p-1} + 1 < 0$, hence $B^{(2)} < 0$.
	In both cases, we have that by Theorem \ref{t:endh}, $h$ cannot be in $\CEND$. 
\end{proof}

\begin{prop}\label{p:END_consec_nonint_check}
	Let $h$ be a vertex of $S_1$. If there exists an index $i$ such that $h_{i} > h_{i+1}$ and $h_{i}, h_{i+1} \notin \ZZ$, then $h \notin \CEND$.
\end{prop}
\begin{proof}
	Let $i$ be such that $h_{i} > h_{i+1}$ and $h_{i}, h_{i+1} \notin \ZZ$ and let $r$ be such that $i \equiv r \pmod p$. 
	Suppose that $r = s$, where $s$ is given in Lemma \ref{l:integer_CM}. Then we have 
	\[CM_s = CM_{s+1} = \dots = CM_{p-1} = CM_0 - 1 = CM_{1} -1 = \dots = CM_{s-1} - 1. \] 
	Since there are two nonzero entries in the $(s+1)$-th mod class, it must be the case that $h_{s+1,CM_{s+1}} = 0$. However, this leads to a contradiction. Since $h$ is a vertex of $S_1$, there must be an entry of 1 in $h$ and let $j$ be the smallest index such that $h_j = 1$. By the minimality of $j$, we have that $h_{j-1} = 0$, and by Proposition \ref{p:END_01_check}, we have arrived at our contradiction. Therefore, we have that $r \neq s$.
	
	Since $h_{i}$ and $h_{i+1}$ are both not 1, there is another nonzero entry somewhere else in the $(r-1)$-th and $r$-th mod classes. Let $\alpha_{r}, \alpha_{r+1}, \omega_{r},$ and $\omega_{r+1}$ be the indices of these nonzero elements where $\alpha_{i} < \omega_{i}$ for all $i$. 
	Since we have two neighboring nonzero elements, without loss of generality, assume $\alpha_{r} = \alpha_{r+1}$ if $r \neq p-1$ and $\alpha_0 = \alpha{p-1} + 1$ otherwise.
 	For ease of notation, we will write $CM$ and $\alpha$.
	
	If $r \neq p-1$, then from the hypothesis that $h_{i} > h_{i+1}$ and Lemma \ref{l:determined_h}, we can obtain
	\begin{align*}
		\frac{\omega_{r} -CM}{\omega_{r} - \alpha} &> \frac{\omega_{r+1} -CM}{\omega_{r+1} - \alpha}\\
		(\omega_{r} - CM)(\omega_{r+1} - \alpha) &> (\omega_{r+1} - CM)(\omega_{r} - \alpha)\\
		-CM\omega_{r+1} - \alpha \omega_{r} &> -CM\omega_{r} - \alpha \omega_{r+1}\\
		0 &> (CM-\alpha)(\omega_{r+1} - \omega_{r})\\
	\end{align*}
	Then from Lemma \ref{l:A2_computation}, 
	\begin{align*}
		A_{r}^{(2)} &= \sum_{j=0}^d (-j)^2 h_{r+1,j} - \sum_{j=0}^d (-j)^2 h_{r,j}\\
		&= (\omega_{r+1} + \alpha) CM - \alpha\omega_{r+1} - (\omega_{r} + \alpha) CM + \alpha\omega_{r}\\
		&= (CM - \alpha)(\omega_{r+1} - \omega_{r}),
	\end{align*}
	and from the work directly above, we have that $A_{r}^{(2)} < 0$. Now, let $r = p-1$. Since $s \neq r$, we have $CM_0 = CM_{p-1} + 1$. Thus, from $h_{i} > h_{i+1}$ and Lemma \ref{l:determined_h}, we have 
	\begin{align*}
		\frac{\omega_{p-1} -CM_{p-1}}{\omega_{p-1} - \alpha_{p-1}} &> \frac{\omega_{0} -CM_{0}}{\omega_{0} - \alpha_0}\\
		\frac{\omega_{p-1} -CM_{p-1}}{\omega_{p-1} - \alpha_{p-1}} &> \frac{\omega_{0} - 1 - CM_{p-1}}{\omega_{0} - 1 - \alpha_{p-1}}\\
		0 &> (CM_{p-1}-\alpha_{p-1})(\omega_{0} - \omega_{p-1} - 1)\\
	\end{align*}
	
	Thus,
	\begin{align*}
		B^{(2)} &= \sum_{j=0}^d (-j)^2 h_{0,j} - \sum_{j=0}^d (-j-1)^2 h_{p-1,j}\\
		&= \sum_{j=0}^d (-j)^2 h_{0,j} - \sum_{j=0}^d (-j)^2 h_{p-1,j} - \sum_{j=0}^d 2jh_{p-1,j} - \sum_{j=0}^d h_{p-1,j}\\
		&= (\alpha_{0} + \omega_{0})CM_{0} - \alpha_{0}\omega_{0} - (\alpha_{p-1} + \omega_{p-1})CM_{p-1} + \alpha_{p-1}\omega_{p-1}  \\
		&\phantom{=} \quad - 2\alpha_{p-1}\frac{\omega_{p-1} - CM_{p-1}}{\omega_{p-1} - \alpha_{p-1}} - 2\omega_{p-1}\frac{CM_{p-1} - \alpha_{p-1}}{\omega_{p-1} - \alpha_{p-1}} - 1  \\
		&=(\alpha_{p-1} + \omega_0 + 1)(CM_{p-1} + 1) - (\alpha_{p-1} + 1)\omega_0 - (\alpha_{p-1} + \omega_{p-1})CM_{p-1}\\
		&\phantom{=} \quad + \alpha_{p-1}\omega_{p-1}- 2CM_{p-1} - 1\\
		&= (CM_{p-1} - \alpha_{p-1})(\omega_0 - \omega_{p-1} - 1).
	\end{align*}
	And so, $B^{(2)} < 0$, and therefore, in both cases, we have that $h \notin \CEND$, by Theorem \ref{t:endh}.
\end{proof}

\begin{example}
	Let $d=3$ and $p = 4$. Consider the following vertex of $S_1$, given by
	\[h = (0,\tfrac{2}{3},0,0,\tfrac{1}{2},0,0,1,0,0,1,0,\tfrac{1}{2},\tfrac{1}{3},0,0) \in S_1.\]
	One can compute $CM(h) = (2,2,2,1)$, so $s = 2$. This $h$-vector happens to violate both conditions in Proposition \ref{p:END_01_check} and Proposition \ref{p:END_consec_nonint_check}. 
	For the condition in Proposition \ref{p:END_01_check}, notice there are two pairs of 0 directly followed by 1, namely $(h_6, h_7)$ and $(h_9, h_{10})$. However, $6 \equiv s \pmod p $, and so $i = 9$ is the only index satisfies the conditions of Proposition \ref{p:END_01_check}.
	For the condition in Proposition \ref{p:END_consec_nonint_check}, notice that $h_{12} > h_{13}$ and both entries are nonintegers.
\end{example}

\begin{example}
	Suppose, similar to Example \ref{ex:general_single_decrease}, we want to locate $h$ with $d= 3$ and $p =4$ such that $L_h(18) > L_h(19)$ is the only decrease of $L_h$. However, we would like to restrict $h$ to be a nonnegative vector. We can still use much of Example \ref{ex:general_single_decrease} to construct such an $h$. For the following procedure, we require the single decrease to occur at $N = kp + m$  with $m \in [0,p-2]$, then we can take the construction from Theorem \ref{thm:cnpolyhedral} and construct a nonnegative $h$. The following construction will work for any $m \neq p-1$.

	We claim that 
	\[h = (0,0,0,15,52,52,52,0,0,0,0,61,24,24,24,0)\]
	is such that $L_h(n+1) \geq L_h(n)$ for all $n \neq 18$, and we can see that this is the case as the difference polynomials are
	\begin{align*}
		f_0(t) &= 0, & f_2(t) &= t^2 - 8t +15,\\
		f_1(t)&= 0, & f_3(t) &= 37t^2 + 74t + 37.
	\end{align*}
	
	The key observation is that adding a positive constant to each entry $h_{i,j}$ for a fixed $j$ preserves $f_r(t)$ for all $r \neq p-1$. Let $f'_r(t)$ denote the difference polynomial obtained by adding some positive constant $c$ to each entry $h_{r,k}$ for a fixed $k$. Then
	\begin{align*}
		f'_{r}(t) &=  (c - c) \binom{t+d-k}{d} + \sum_{j = 0}^{d} (h_{r+1,j} - h_{r,j}) \binom{t+d-j}{d} = f_r(t)  \text{ for } r \neq p-1\\
		f'_{p-1}(t) &= c\binom{t+d-k+1}{d} - c\binom{t+d-k}{d} + \sum_{j = 0}^d h_{0,j} \binom{t+d-j+1}{d} - h_{p-1,j} \binom{t+d-j}{d}\\
		&= c\binom{t+d-k}{d-1} + f_{p-1}(t)
	\end{align*} 
	The last line utilizes a binomial identity. It remains to show that if $f_{p-1}(t) \geq 0$, then $f'_{p-1}(t) \geq 0$ for all $t \in \ZZ_{\geq 0}$. This is the case as $k \leq d$, and so $\binom{t+d-k}{d-1} \geq 0$ for all $t$. Therefore, this process can take any $h$ such that $L_h$ experiences a single decrease in the $r$-th mod class with $r \neq p-1$ and construct a nonnegative vector $h'$ such that $L_{h'}$ experiences a single decrease in the same location as $L_h$.
\end{example}

\section{Growth rates for the number of nonnegative eventually nondecreasing quasi-polynomials}\label{sec:growth}

Our final goal in this work is to determine the growth rate for the number of eventually nondecreasing quasi-polynomials with non-negative integer $h$-vectors as a function of their volume.
Because $C_{\textnormal{END}}\cap \RR_{\geq 0}^{p(d+1)}$ is a union of finitely many relatively open polyhedral cones, the dominant growth term in our enumeration will come from the relative interior of $\overline{\CEND}$, i.e., from those eventually non-decreasing quasi-polynomials that strictly satisfy the first-level inequalities.
This makes sense from a topological perspective, as we can consider a vector $h\in C_{\textnormal{END}}\cap \RR_{\geq 0}^{p(d+1)}$ to be generic if and only if it is contained in the relative interior of $ \overline{\CEND}\cap \RR_{\geq 0}^{p(d+1)}$.
This motivates the following definition.

\begin{defn}\label{d:fixedvolume}
Fix positive integers $d$ and $p$.
We say that a vector $h\in C_{\textnormal{END}}$ is \emph{generic} if it satisfies the volume inequalities, strictly satisfies the first-level inequalities, and has all positive entries.
Let $\nd(d,p;V)$ denote the number of vectors of volume $V\in \ZZ_{\geq 0}$ in $C_{\textnormal{END}}\cap \ZZ_{\geq 0}^{p(d+1)}$, and let $\gnd(d,p;V)$ denote the number of generic vectors of volume $V\in \ZZ_{\geq 0}$ in $C_{\textnormal{END}}\cap \ZZ_{\geq 0}^{p(d+1)}$.
\end{defn}

Note that the values $\nd(d,p;V)$ and $\gnd(d,p;V)$ exists since there are only finitely many positive integer vectors with sum $V$.

\begin{thm}\label{thm:enumeration}
Fix positive integers $d$ and $p$.
The function $\gnd(d,p;V)$ is a quasi-polynomial of degree $pd$ and having period dividing $\lcm(1,2,\ldots,d)$, non-negative $h$-vector, and volume given by the normalized volume of $S_1$.
The function $\nd(d,p;V)$ is a quasi-polynomial of degree $pd$.
\end{thm}

\begin{proof}
Observe that $h$ is generic if and only if it is an integer point contained in the interior of $ \overline{\CEND}\cap \RR_{\geq 0}^{p(d+1)}$.
Any such $h$ has a volume $V$, and therefore it is contained in $V\cdot S_1^\circ$, the $V$-th dilation of the relatively open polytope $S_1^\circ$.
Further, any $h$ in the relative interior of $V\cdot S_1$ is generic, since $V\cdot S_1$ is the intersection of $ \overline{\CEND}\cap \RR_{\geq 0}^{p(d+1)}$ by a hyperplane and hence relative interiors agree between the cone and the intersection.
Thus, we have that $\gnd(p,d;V)$ is the number of integer points in $V\cdot S_1^\circ$.

Observe that the dimension of $S_1$ is $pd$ and every entry of the vertices of $S_1$ have denominators from the set $\{1,2,\ldots,d\}$.
Further, every value in $\{1,2,\ldots,d\}$ appears as a denominator of some vertex entry.
Combining the fundamental theorem of Ehrhart theory for rational polytopes with Ehrhart-MacDonald reciprocity~\cite{ccd}, we have that $\gnd(p,d;V)$ is a quasi-polynomial of degree equal to the dimension of $S_1$ and having period dividing the least common multiple of the denominators of the vertices of $S_1$, with non-negative $h$-vector and volume given by the normalized volume of $S_1$.
This completes the proof of the claim regarding $\gnd(d,p;V)$.

For the case of $\nd(d,p;V)$, by Corollary~\ref{cor:disjointend}, we have that $S_1\cap \CEND$ is a disjoint union of finitely many partially-open rational polytopes of dimension at most $dp$ (where \emph{partially-open} means obtained as an intersection of both closed and open half-spaces).
Thus, $\nd(d,p;V)$ is the sum of the Ehrhart quasi-polynomials of the partially-open rational polytopes in the disjoint union, of which $S_1^\circ \cap \CEND$ has dimension $dp$.
Thus, $\nd(d,p;V)$ is a quasi-polynomial of degree $dp$, completing the proof.
\end{proof}

\begin{example}
	The following are the numerators of the rational forms for the ordinary generating functions of $\gnd(d,p;V)$ for various $d$ and $p$, where the denominator of the rational form is $(1-t^{\lcm(1,\ldots,d)})^{pd+1}$.
	\[
	\begin{array}{c|c|l}
		d & p  & \text{Numerator} \\
		\hline & & \\
		2 & 2 & t^{10} + 5t^{9} + 14t^{8} + 25t^{7} + 26t^{6} + 14t^{5} + 3t^{4} \\
		2 & 3 & t^{14} + 7t^{13} + 41t^{12} + 139t^{11} + 286t^{10} + 389t^{9} + 358t^{8} + 217t^{7} + 81t^{6} + 16t^{5} + t^{4} \\
		2 & 4 & t^{18} + 9t^{17} + 100t^{16} + 525t^{15} + 1784t^{14} + 4185t^{13} + 6934t^{12} + 8312t^{11} + 7259t^{10} \\ 
		&  & \phantom{aaaaaaaaaaaaaaa} + 4561t^{9} + 2010t^{8} + 591t^{7} + 104t^{6} + 9t^{5}\\
		3 & 2 & t^{42} + 7t^{41} + 42t^{40} + 166t^{39} + 507t^{38} + 1297t^{37} + 2915t^{36} + 5918t^{35} + 10998t^{34} \\
		& & \phantom{aaaaaaaaaaaaaa}+ 18929t^{33} + 30430t^{32} + 45996t^{31} + 65674t^{30} + 88878t^{29} \\
		& & \phantom{aaaaaaaaaaaaaa} + 114349t^{28} + 140180t^{27} + 164021t^{26} + 183384t^{25} \\ 
		& & \phantom{aaaaaaaaaaaaaa}+ 196083t^{24} + 200624t^{23} + 196455t^{22} + 184076t^{21} \\
		& & \phantom{aaaaaaaaaaaaaa}+ 164949t^{20} + 141248t^{19} + 115444t^{18} + 89905t^{17} + 66561t^{16} \\
		& & \phantom{aaaaaaaaaaaaaa}+ 46714t^{15} + 30968t^{14} + 19299t^{13} + 11234t^{12} + 6054t^{11} \\
		& & \phantom{aaaaaaaaaaaaaa} + 2987t^{10} + 1327t^{9} + 517t^{8} + 168t^{7} + 41t^{6} + 6t^{5}	
	\end{array}
	\]
\end{example}

\begin{example}
	The following are the rational forms for the ordinary generating functions of $\nd(d,p;V)$ for various $d$ and $p$..
	\[
	\begin{array}{c|c|l}
		d & p  & \text{Rational generating function for $\nd(d,p;V)$}\\	
		\hline & & \\
		2 & 2 & (t^{17} + 6t^{16} + 21t^{15} + 55t^{14} + 114t^{13} + 197t^{12} + 291t^{11} + 371t^{10} + 414t^9 + 405t^8 \\
		&  & \phantom{aaaaaaaaaaa} + 347t^7 + 261t^6 + 170t^5 + 95t^4 + 45t^3 + 17t^2 + 5t + 1)/(1-t^4)^5  \\
		& & \\
		2 & 3 & (t^{25} + 10t^{24} + 56t^{23} + 220t^{22} + 658t^{21} + 1621t^{20} + 3409t^{19} + 6244t^{18} + 10137t^{17} \\
		&  & \phantom{aaaaaaaaaaa} + 14729t^{16} + 19280t^{15} + 22856t^{14} + 24600t^{13} + 24066t^{12} \\
		&  & \phantom{aaaaaaaaaaa} + 21402t^{11} + 17272t^{10} + 12607t^9 + 8284t^8 + 4864t^7 + 2524t^6 \\
		&  & \phantom{aaaaaaaaaaa} + 1142t^5 + 441t^4 + 141t^3 + 36t^2 + 7t + 1)/(1-t^4)^7 \\
		& & \\
		3 & 2 & (6t^{15} + 21t^{14} + 42t^{13} + 67t^{12} + 88t^{11} + 95t^{10} + 101t^9 + 102t^8\\
		&  & \phantom{aaaaaaaaaaa} + 98t^7 + 96t^6 + 80t^5 + 58t^4 + 34t^3 + 14t^2 + 3t + 1))/ \\
		&  & \phantom{aaaaaaaaaaa} (t^6 + t^3 + 1)(t^2 + t + 1)^2(1 + t^2)(1 + t)(1 -t )^7
	\end{array}
	\]
\end{example}

\begin{remark}
We used the following process to compute the generating function for $\nd(d,p;V)$.
First, we considered the disjoint union decomposition of $S_1$ obtained by intersecting $S_1$ with each of the chambers of the disjoint union decomposition of $\CEND$ from Corollary~\ref{cor:disjointend}.
Each such chamber is indexed by the set of indices $\ell_r$ for $r=0,\ldots,p-1$ where $A_r^{(\ell_r)}$ or $B^{(\ell_{p-1})}$ is first strictly positive, such that the minimum value of the sequence $(\ell_r)$ is $1$; the minimum condition is due to the fact that if $A_r^{(1)}=0$ for all $r$ and $B^{(1)}=0$, then the only $h$-vector satisfying those equations is the origin.
For each of these chambers, we computed the set of faces for the closure of the chamber.
For each such face $F$, we computed an interior point of $F$ and checked to see if it was contained in any of the hyperplanes $A_r^{(\ell_r)}=0$ or $B^{(\ell_{p-1})}=0$.
If so, we discarded $F$.
For all remaining $F$, we computed the rational generating function for the Ehrhart series of $F^\circ$, the interior of $F$, and added them to a list.
Finally, we summed all the resulting Ehrhart series, plus $1$ corresponding to the origin in the cone over $S_1$, and the result is the generating function for $\nd(d,p;V)$.
\end{remark}

\section{Proof of Theorem \ref{t:S1_vertex}}\label{sec:proofs}
\label{sec:proof}

The purpose of this section is to establish that any vector failing one of the three conditions in Theorem~\ref{t:S1_vertex} is not a vertex of $S_1$.
We begin this section by providing the three statements necessary for the corresponding contrapositive arguments.

\begin{lemma}\label{l:max}
	Let $h \in S_1$. If $\max(h) < 1$, then $h$ satisfies at most $pd-1$ facet-defining inequalities of $S_1$.
	Moreover, $S_1$ is $pd$-dimensional and $h$ is not a vertex of $S_1$.
\end{lemma}

\begin{prop}\label{p:2nonzero_decomp}
	Any point $h \in S_1$ with at least three nonzero entries of the form $h_{r,j}$ for a fixed $r \in [0,p-1]$ can be written as a convex combination of distinct points in $S_1$, and hence $h$ cannot be a vertex of $S_1$.
\end{prop}

\begin{prop}\label{p:nonintegerCM}
	Any point $h \in S_1$ such that $CM(h) \notin \ZZ^{p}$ can be written as a convex combination of distinct points in $S_1$, and hence $h$ cannot be a vertex of $S_1$.
\end{prop}

The proofs of Proposition~\ref{p:2nonzero_decomp} and~\ref{p:nonintegerCM} can be combined to produce an algorithm that will take any $h \in S_1$ not satisfying conditions (2) or (3) in Theorem~\ref{t:S1_vertex} and represent $h$ as a convex combination of distinct points in $S_1$, thus showing $h$ is not a vertex. 
An example illustrating this is given below, in the hope that it will clarify some of the technical details of the proofs of these Propositions.

\begin{example}\label{ex:runningexample}
	Let $d = 3$ and $p = 3$ and consider 
	\[
		g = \begin{bmatrix}
			0 & 0.1 & 0.1 & 0.3 & 0.3 & 0 & 0 & 0 & 0.9 & 0.7 & 0.6 & 0
		\end{bmatrix} \, .
	\]
	
	Note that $g$ violates conditions (2) and (3) in Theorem \ref{t:S1_vertex}. Then, following the procedure described in the proof of Proposition \ref{p:2nonzero_decomp}, we find that $g$ can be written as the convex sum $\tfrac{1}{3}x + \tfrac{2}{3}y$, where 
	\begin{align*}
		x &= \begin{bmatrix}
			0 & 0.3 & 0.1 & 0.3 & 0 & 0 & 0 & 0 & 0.9 & 0.7 & 0.7 & 0
		\end{bmatrix},\\
		y &= \begin{bmatrix}
			0 & 0& 0.1 & 0.3 & 0.45 & 0 & 0 & 0 & 0.9 & 0.7 & 0.55 & 0
		\end{bmatrix}.
	\end{align*}
	
	Both $x$ and $y$ violate condition (3), so following the procedure described in the proof of Proposition \ref{p:nonintegerCM}, we find that 
	\begin{align*}
		x &= \frac{1}{10}\begin{bmatrix}
			0 & 0 & 0 & 0 & 0 & 0 & 0 & 0 & 1 & 1 & 1 & 0
		\end{bmatrix} & 
		y &= \frac{1}{10}\begin{bmatrix}
			0 & 0 & 0 & 0 & 0 & 0 & 0 & 0 & 1 & 1 & 1 & 0
		\end{bmatrix}
		\\
		&+ \frac{3}{10}\begin{bmatrix}
			0 & \frac{1}{3} & 0 & 0 & 0 & 0 & 0 & 0 & 1 & 1 & \frac{2}{3}  & 0
		\end{bmatrix} & 
		&+ \frac{3}{10}\begin{bmatrix}
			0 & 0 & 0 & 0 & \frac{1}{2} & 0 & 0 & 0 & 1 & 1 & \frac{1}{2} & 0
		\end{bmatrix}
		\\
		&+ \frac{4}{10}\begin{bmatrix}
			0 & \frac{1}{3} & 0 & \frac{1}{2}  & 0 & 0 & 0 & 0 & 1 & \frac{1}{2} & \frac{2}{3} & 0
		\end{bmatrix} &
		&+ \frac{4}{10}\begin{bmatrix}
			0 & 0 & 0 & \frac{1}{2} & \frac{1}{2} & 0 & 0 & 0 & 1 & \frac{1}{2} & \frac{1}{2} & 0
		\end{bmatrix} 
		\\
		&+ \frac{2}{10}\begin{bmatrix}
			0 & \frac{1}{3} & \frac{1}{2} & \frac{1}{2} & 0 & 0 & 0 & 0 & \frac{1}{2} & \frac{1}{2} & \frac{2}{3} & 0
		\end{bmatrix} &
		&+ \frac{2}{10}\begin{bmatrix}
			0 & 0 & \frac{1}{2} & \frac{1}{2} & \frac{1}{2} & 0 & 0 & 0 & \frac{1}{2} & \frac{1}{2} & \frac{1}{2} & 0
		\end{bmatrix}
	\end{align*}
	Note that each vector in the decomposition of $x$ and $y$ satisfies the volume equalities and the first-level inequalities. In addition, the last vector does not satisfies condition (1) in Theorem \ref{t:S1_vertex}; however, each of those vectors are still in $\overline{\CEND}$, which is sufficient as we only need a convex combination of points in $S_1$, not vertices in $S_1$.
	
	If one is interested in an algorithm to fully decompose a point in $S_1$ into vertices, then one must develop an algorithm that takes an input a vector of $S_1$ satisfying conditions (2) and (3), but not condition (1) and produce a convex combination into vertices of $S_1$.
\end{example}

\begin{proof}[Proof of Lemma~\ref{l:max}]
	Suppose that $\max(h) < 1$. 
	Since the sum of the terms in each mod class is 1, there must be at least two nonzero entries in each mod class. 
	Thus, there are most $(d-1)p$ zero entries in $h$. However, the first and last first-level inequalities, namely
	\[
	-1 + CM_0 \leq CM_{p-1} \quad \text{and} \quad CM_{1} \leq CM_0
	\]
	imply that at least one of the first-level inequalities cannot be satisfied with equality. 
	Together with the volume equations, $h$ can therefore satisfy at most $pd - 1$ inequalities with equality.
	It suffices to show that $S_1$ is a $pd$-dimensional polyhedron, as then $h$ cannot be a vertex.	
	
	By the volume equations, we can consider $S_1$ as the intersection between the hyperplane $\sum_{i=0}^{dp+p-1} h_i = p$ and $\overline{\CEND}$. From Corollary \ref{c:endclosure}, we know that $\overline{\CEND}$ is a $(pd+1)$-dimensional polyhedron and one can show that
	\[\left(\tfrac{1}{d+1}, \tfrac{1}{d+1}, \dots, \tfrac{1}{d+1} \right) \in S_1,\]
	and this point lies in the interior of $\overline{\CEND}$. Therefore, the hyperplane is not redundant on $\overline{\CEND}$ and so the intersection reduces the dimension by one.
	Therefore, $S_1$ has codimension 1 and has dimension $dp$, as desired.
\end{proof}

We will first show that if $h \in S_1$ does not contain an entry of 1, i.e., $h$ violates condition (1) in Theorem \ref{t:S1_vertex}, then it cannot be a vertex.

\begin{proof}[Proof of Proposition~\ref{p:2nonzero_decomp}]
	Let $h \in S_1$ and suppose that $h_{r,j}$ is nonzero for at least three distinct values of $j$. 
	By the volume equations, we have $\sum_{j=0}^dh_{r,j} = 1$ for all $r$.
	Thus, using the definition of $CM_r$, we have 
	\[
	\sum_{j=0}^d j h_{r,j} = CM_r(h) \qquad \text{ if and only if } \qquad \sum_{j=0}^d (j-CM_r(h))h_{r,j} = 0\, .
	\]
	We would like to decompose $h$ into a convex combination of vectors $x,y$ such that
	\[h_{i,j} = x_{i,j} = y_{i,j} \quad \text{for all } j \text{ and } i \neq r.\]
	So we will only focus on the $r$-th mod class. If $x$ or $y$ have at least three nonzero entries in some mod class, the following algorithm can be repeated.
	
	Let $\alpha$ and $\omega$ be the smallest and largest value respectively, such that $h_{r,\alpha} \neq 0$ and $h_{r,\omega} \neq 0$.
	Let us construct $x$ such that $x_{i,j} = h_{i,j}$ for all $j$ and for all $i \neq r$ and
	\[x_{r,\alpha} = \frac{\omega - CM_r(h)}{\omega - \alpha}, \quad x_{r,\omega} = \frac{CM_r(h) - \alpha}{ \omega - \alpha},\]
	and $x_{r,j} = 0$ for all $j \neq \alpha, \omega$. 
	Note that $CM_r(x) = CM_r(h)$ and $x$ satisfies the volume equations, thus $x \in \overline{\CEND}$.
	
	Let $\delta := \min\{(CM_r(h) - \alpha) h_{r, \alpha} , (\omega - CM_r(h)) h_{r,\omega}\}$ and let
	\[\lambda := \left(\frac{\delta}{CM_r(h) - \alpha} + \frac{\delta}{\omega - CM_r(h)}\right).\]
	Since $h_{r,j}$ has multiple nonzero entries, $(CM_m(h)-\alpha)h_{m,\alpha} > 0$ by equation (\ref{eq:CM_balance}). Thus, $\delta$ and $\lambda$ are positive. 
	Also, since there is a third nonzero entry of $h_{m,j}$,
	\[0 < \lambda = \frac{\delta}{CM_m(h)- \alpha} + \frac{\delta}{\omega - CM_m(h)} \leq  h_{m,\alpha} + h_{m,\omega} < 1.\]
	Let $y := \frac{1}{1-\lambda} (h -  \lambda x).$
	If we can show that $y \in \overline{\CEND}$, then we are done, as $h$ can be written as the convex sum $h = \lambda x + (1-\lambda) y$.
	Without loss of generality, let $j = \alpha$ and consider the difference
	\[h_{r,\alpha} - \lambda x_{r,\alpha} = h_{r,\alpha} - \lambda \frac{\omega - CM_r(h)}{\omega - \alpha} = h_{r,\alpha} - \frac{\delta}{CM_r(h) - \alpha} \geq h_{r,\alpha} - h_{r,\alpha} = 0.\]
	Thus, $y$ is a nonnegative vector. 
	Observe that 
	\[\sum_{j=0}^d (j-CM_r(h)) (h -  \lambda x) = \sum_{j = 0}^d (j-CM_r(h))h_{r,j} - \lambda\sum_{j=0}^d (j-CM_r(x))x_{r,j} = 0 - 0.\]
	Therefore, $CM_r(y) = CM_r(h)$ and to show that $y \in \overline{\CEND}$, it remains to show that $y$ satisfies the volume equations. For all $i \in [0,p-1]$, we observe
	\[\sum_{j = 0}^d y_{i,j} = \frac{1}{1-\lambda}\left(\sum_{j=0}^d h_{i,j} - \lambda\sum_{j=0}^d x_{i,j}\right) =\frac{1-\lambda}{1-\lambda} = 1 \, ,\]
	and hence our proof is complete.
\end{proof}

In Proposition~\ref{p:nonintegerCM} below, we will work with noninteger $CM$ vectors. 
The purpose of the proposition is to establish a process to decompose any $h\in S_1$ with a noninteger $CM$ vector into a convex combination of vectors $v_i \in \overline{\CEND}$ with integer $CM$ vectors. 
We will begin by considering how rounding interacts with the first-level inequalities.

We would like to round noninteger $CM$ vectors to have integer values which satisfy the first-level inequalities. 
This restriction forces a very particular rounding procedure. 
For example, if $CM(h)=(2.4,2.1,1.8),$ then the vector $(3,2,1)$ cannot occur as the $CM$ vector of an element of $\overline{\CEND}$, since it violates the first-level inequalities. 
In particular, $1.8$ can be rounded down only if both $2.4$ and $2.1$ are rounded down.
This motivates the following definition; note that Lemma~\ref{l:rounding_noninteger_CM} below proves that Definition~\ref{def:cmround} is well-defined.

\begin{defn}\label{def:cmround}
	For an integer $i$, let $\wt{CM}^{(i)}\in \ZZ^p$ denote the unique integer vector obtained from $CM$ by rounding down exactly $i$ noninteger coordinates of $CM$ and rounding up all remaining noninteger coordinates, subject to the condition that $\wt{CM}^{(i)}$ satisfy the first-level inequalities given by~\eqref{eq:first_order_CM}.
\end{defn} 

\begin{example}
	Let $d \geq 3$ and $p = 3$. Suppose $h \in \overline{\CEND}$ is such that $CM(h) = (2.4, 2.1, 1.8)$. Then
	\[
	\widetilde{CM}^{(0)}=(3,3,2), \quad  \widetilde{CM}^{(1)}=(3,2,2), \quad 
	\widetilde{CM}^{(2)}=(2,2,2), \quad  \widetilde{CM}^{(3)}=(2,2,1) \, .
	\]
\end{example}

Recall from Lemma \ref{l:integer_CM} that if $CM \in \ZZ^p$, then we have that $CM = (k+1,k+1,\dots,k+1, k,k,\dots,k)$ for some $k \in \ZZ$ with $s+1$ many values of $k+1$ appearing, where $s$ is the offset index.
In addition, for each positive $i$, the vector $\wt{CM}^{(i)}$ can be obtained from $\wt{CM}^{(i-1)}$ by decreasing exactly one coordinate by one. 
We make these observations precise in the following lemma, which also proves that $\wt{CM}^{(i)}$ is well-defined.

\begin{lemma}\label{l:rounding_noninteger_CM}
	Given $CM \in \RR^p$ satisfiying the first-level inequalities, let $s$ be the (unique) index such that $\lfloor CM_{s} \rfloor > \lfloor CM_{s+1} \rfloor$. If $s$ does not exists, i.e. $\lfloor CM_{r} \rfloor = \lfloor CM_{r+1} \rfloor$ for all $r \in [0,p-2]$, we set $s = p-1$. Let $(s_1, s_2, \dots, s_n)$ be the ordered list of indices $r$ such that $CM_r \notin \ZZ$, obtained from the cyclic sequence $s ,s-1, s-2, \dots, s-(p-2) \pmod p$ by deleting all indices such that $CM_{s-i} \in \ZZ$, i.e., $(s_1, s_2, \dots, s_n)$ is the subsequence of 
	\[
	(s-1,s-2,\ldots,1,0,p-1,p-2,\ldots,s) 
	\]
	obtained by deleting indices with integer coordinates of $CM$.
	The coordinates which are rounded down in $\wt{CM}^{(i)}$ are precisely $(s_1, s_2, \dots s_i)$, and rounding down any other set of $i$ entries of $CM$ violates~\eqref{eq:first_order_CM}.
\end{lemma}

\begin{proof}
	We first show that $s$ is unique. Let $s$ be the maximal index such that $\lfloor CM_{s} \rfloor > \lfloor CM_{s+1} \rfloor$. Now suppose that there exists another $s'$ such that $\lfloor CM_{s'} \rfloor > \lfloor CM_{s'+1} \rfloor$. By maximality and the first-order inequalities, we have 
	\[\lfloor CM_{s'} \rfloor > \lfloor CM_{s'+1} \rfloor \geq \lfloor CM_{s} \rfloor > \lfloor CM_{s+1} \rfloor . \]
	Since these quantities are integers, we can then conclude that 
	\[
	\lfloor CM_{s'} \rfloor \geq \lfloor CM_{s'+1}\rfloor + 1 \geq \lfloor CM_{s}\rfloor  + 1 \geq \lfloor CM_{s+1}\rfloor + 2.\]
	However, this is a contradiction, as the first-level inequalitity $CM_{s+1} \geq CM_{s'} - 1$ and $CM_{s+1}\notin \ZZ$ implies that
	\[
	\lfloor CM_{s+1} \rfloor + 2  > CM_{s+1}+1 >  \lfloor CM_{s'+1} +1 \rfloor  \geq \lfloor CM_{s}\rfloor 
	\]
	Thus, either $s$ exists and is unique or no such $s$ exists, in which case, we set $s = p-1$.
	
	We now show that in order to result in a vector satisfying~\eqref{eq:first_order_CM}, any $i$ coordinates rounded down in $CM$ must form an initial segment of $(s_1, s_2, \dots, s_n)$ if the remaining coordinates in $CM$ are rounded up. 
	Suppose that there exist indices $a < b$ such that the $(s_a)$-th coordinate is rounded up while the $(s_b)$ coordinate is rounded down. 
	We will show that if $s_a > s_b$, then $\lceil CM_{s_a} \rceil > \lfloor CM_{s_b} \rfloor$ and if $s_a < s_b$, then $\lceil CM_{s_a} \rceil > \lfloor CM_{s_b} \rfloor + 1$. 
	These lead to a contradiction, as the resulting vector does not satisfy the first-level inequalities. 
	
	In the first case, we have $s_a > s_b$. By definition of $s$, we have $\lfloor CM_{s_a} \rfloor = \lfloor CM_{s_b} \rfloor$.
	Since $CM_{s_a} \notin \ZZ$, we have $\lceil CM_{s_a} \rceil = \lfloor CM_{s_a} + 1 \rfloor$, and so
	\[\lceil CM_{s_a} \rceil = \lfloor CM_{s_a} + 1 \rfloor > \lfloor CM_{s_b} \rfloor \, .\]
	
	Otherwise, we have $s_a < s_b$. This can only happen in the case where if $s_a \equiv s- i \pmod p$ and $s_b \equiv s-j \pmod p$, then $i < s < j$. In particular, we have $s_a < s < s_b$. Thus, by definition of $s$, we have 
	\[\lfloor CM_{s_a} \rfloor = \lfloor CM_{s} \rfloor > \lfloor CM_{s_b} \rfloor \, , \]
	which then implies
	\[\lceil CM_{s_a} \rceil = \lfloor CM_{s_a} + 1 \rfloor > \lfloor CM_{s_b} + 1 \rfloor = \lfloor CM_{s_b} \rfloor + 1.\]
\end{proof}

\begin{proof}[Proof of Proposition \ref{p:nonintegerCM}]
	Let $n$ denote the number of noninteger entries in $CM=CM(h)$. 
	Then let $A \in \ZZ_{\geq 0}^{(p+1) \times (n+1)}$ and $\vec b \in \RR^{p+1}$ be given by
	\[A = \begin{bmatrix}
		\widetilde{CM}^{(0)} & \widetilde{CM}^{(1)} & \dots & \widetilde{CM}^{(n)} \\
		1 & 1 & \dots & 1
	\end{bmatrix}, \qquad \vec b = \begin{bmatrix}
		CM \\ 1
	\end{bmatrix} \]
	and consider the linear system corresponding to the augmented matrix $[A \mid \vec{b}]$. We claim that $[A \mid \vec{b}]$ is row-equivalent to
	\[\left[\begin{array}{cccccc|c}
		1 & 0 & 0 & \dots & 0 & 0 & q_{s_1}\\
		1 & 1 & 0 & \dots & 0 & 0 & q_{s_2}\\
		1 & 1 & 1 & \dots & 0 & 0 & q_{s_3}\\
		\vdots & \vdots & \vdots & \vdots & \ddots & \vdots & \vdots\\
		1 & 1 & 1 & \dots & 1 & 0 & q_{s_{n}}\\
		1 & 1 & 1 & \dots & 1 & 1 & 1\\
		0 & 0 & 0 & \dots & 0 & 0 & 0 \\
		\vdots & \vdots & \vdots & \vdots & \ddots & \vdots & \vdots\\
		0 & 0 & 0 & \dots & 0 & 0 & 0
	\end{array}\right]\]
	where $(s_1, s_2, \dots, s_n)$ is defined as in Lemma \ref{l:rounding_noninteger_CM} and $q_r := CM_{r} - \lfloor CM_{r} \rfloor$. 
	First, observe that if $CM_r \in \ZZ$, then the $r$-th row of $[A \mid \vec{b}]$ is constant. 
	Thus, the last row of $[A \mid \vec{b}]$ can be used to eliminate the $r$-th row. 
	So, we can reduce $[A \mid \vec{b}]$ to the case where the only nonzero rows consist of those corresponding to the noninteger coordinates of $CM$, together with the final row of $1$'s. 
	By Lemma \ref{l:rounding_noninteger_CM}, the $s_i$-th row of $[A \mid \vec b]$ is 
	\[
	[
	\underbrace{
		\begin{array}{ccc}
			\lceil CM_{s_i} \rceil & \dots & \lceil CM_{s_i} \rceil 
		\end{array}
	}_{i \text{ many}}
	\begin{array}{ccc|c}
		\lfloor CM_{s_i} \rfloor & \dots & \lfloor CM_{s_i} \rfloor & CM_{s_i}
	\end{array}
	]
	\]
	and by subtracting $\lfloor CM_{s_i} \rfloor$ copies of the last row from the $s_i$-th row, we obtain the row
	\[
	[
	\underbrace{
		\begin{array}{ccc}
			1 & \dots & 1 
		\end{array}
	}_{i \text{ many}}
	\begin{array}{ccc|c}
		0 & \dots & 0 & q_{s_i}
	\end{array}
	] \, .
	\]
	
	Thus, through row operations, we can arrive at the desired matrix, and so the solution to $[A \mid \vec{b}]$ is unique and is given by 
	\[
	(\lambda_1 ,\lambda_2, \dots, \lambda_n, \lambda_{n+1}) = (q_{s_1} ,q_{s_2} - q_{s_1} , q_{s_3} - q_{s_2}, \dots, q_{s_n} - q_{s_{n-1}}, 1 - q_{s_n} ) \, .
	\]
	To show that this solution is nonnegative, assume that $i \in [1,n-1]$ is such that $s_{i+1} < s_i$. 
	Then by definition of $s$, we have $\lfloor CM_{s_i} \rfloor = \lfloor CM_{s_{i+1}} \rfloor$, and since $CM_{s_{i+1}} \geq CM_{s_i}$, we have $q_{s_{i+1}} \geq q_{s_i}$. Otherwise, $s_{i+1} > s_i$, and in particular $s_i < s < s_{i+1}$. 
	Thus, by definition of $s$, $\lfloor CM_{s_i} \rfloor \neq \lfloor CM_{s_{i+1}} \rfloor$, and from $CM_{s_i} \geq CM_{s_{i+1}} \geq CM_{s_i} - 1$ we have that $\lfloor CM_{s_{i+1}} \rfloor = \lfloor CM_{s_{i}} \rfloor - 1$. Thus,
	\begin{align*}
		CM_{s_{i+1}} +1 &\geq CM_{s_i}\\
		CM_{s_{i+1}} + 1 - \lfloor CM_{s_{i+1}} \rfloor &\geq CM_{s_i} - \lfloor CM_{s_{i}} \rfloor + 1\\
		q_{s_{i_1}} & \geq q_{s_i}.
	\end{align*} 
	Also, note that by definition of $s_i$, $q_{s_1} \neq 0$ and $q_{s_n} \neq 1$. Therefore, we have $0 < q_{s_1} \leq q_{s_2} \leq \dots \leq q_{s_n} < 1$. Thus, our solution is nonnegative and as $n \geq 1$ by assumption, at least the first and last coordinates of the solution are nonzero.
	
	By Proposition \ref{p:2nonzero_decomp}, we can assume that $h$ has at most two nonzero entries in each mod class, as otherwise $h$ is not a vertex. 
	We now claim that the solution above provides a decomposition of $h$ into vectors in $\overline{\CEND}$. By Lemma \ref{l:determined_h}, there exists a vector $v_i$ in $\overline{\CEND}$ with center-of-mass vector $\wt{CM}^{(i)}$ such that for each coordinate where $h$ is 0, then that coordinate is also 0 for $v_i$. 
	Now consider the linear combination
	\[\sum_{i = 1}^{n+1} \lambda_i v_i.\]
	Since $(\lambda_1, \dots, \lambda_{n+1})$ gives a solution to the above system, $\sum_{i = 1}^{n+1} \lambda_i v_i$ has the same $CM$ vector as $h$ and by construction, the entries of $\sum_{i = 1}^{n+1} \lambda_i v_i$ that can be nonzero are the same as $h$. Therefore, by Lemma \ref{l:determined_h}, we have 
	\[h = \sum_{i = 1}^{n+1} \lambda_i v_i.\]
	Thus, as the sum of $\lambda_i$ is 1, we have that $h$ is a nontrivial convex sum of vectors in $\overline{\CEND}$. Thus, $h$ cannot be a vertex.
\end{proof}

\section{Tool and Computational Resource Disclosure}

The authors did not use AI or LLM tools for any aspect of this research or the writing of this manuscript.
SageMath~10.9, available at \url{http://sagemath.org}, was used for all computations in support of this project.


\end{document}